\documentclass[reqno]{amsart}

\usepackage{amsmath,amsthm,amsfonts,amssymb,young}
\usepackage{array, boldline, makecell, booktabs}
\usepackage{bm}
\usepackage{ytableau}
\usepackage{euscript, mathrsfs} 
\usepackage{tikz}
\usetikzlibrary{backgrounds}
\usepackage{amsmath}
\usepackage{mathtools}
\usepackage[colorlinks]{hyperref}
\usepackage{cleveref}
\usepackage[margin=1in]{geometry}
\numberwithin{equation}{section}    
\usepackage[all]{xy}
\usepackage{graphicx,import}
\usepackage{amscd}
\usepackage{color}
\usepackage{enumerate}
\usepackage{fourier}
\usepackage{tikz}
\usepackage{cases}
\usepackage{cite}
\usepackage[T1]{fontenc}

\newtheorem{thm}{Theorem}[section]
\newtheorem{lem}[thm]{Lemma}
\newtheorem{proposition}[thm]{Proposition}
\newtheorem{corollary}[thm]{Corollary}
\theoremstyle{definition}
\newtheorem{example}[thm]{Example}
\newtheorem{definition}[thm]{Definition}
\newtheorem{conj}[thm]{Conjecture}
\newtheorem{rmk}[thm]{Remark}

\newtheorem{problem}[thm]{Problem}

\newcommand*{\vertbar}{\rule[-1ex]{0.5pt}{2.5ex}}

\DeclareMathOperator{\RD}{RD}
\DeclareMathOperator{\I}{I}

\DeclareMathOperator{\Rec}{Rec}
\DeclareMathOperator{\Sub}{Sub}
\DeclareMathOperator{\rsum}{rsum}
\DeclareMathOperator{\csum}{csum}
\DeclareMathOperator{\LW}{LW}
\DeclareMathOperator{\bo}{bot}
\DeclareMathOperator{\piv}{piv}

\DeclareMathOperator{\area}{area}
\DeclareMathOperator{\grFrob}{grFrob}

\DeclareMathOperator{\Inv}{Inv}

\DeclareMathOperator{\cont}{cont}

\DeclareMathOperator{\adj}{adj}

\DeclareMathOperator{\iDes}{iDes}
\DeclareMathOperator{\Des}{Des}

\DeclareMathOperator{\std}{std}

\DeclareMathOperator{\maj}{maj}

\DeclareMathOperator{\arm}{arm}
\DeclareMathOperator{\leg}{leg}
\DeclareMathOperator{\coleg}{coleg}
\DeclareMathOperator{\coarm}{coarm}
\DeclareMathOperator{\inv}{inv}
\DeclareMathOperator{\dinv}{dinv}

\DeclareMathOperator{\st}{st}

\DeclareMathOperator{\stat}{stat}
\DeclareMathOperator{\cycling}{Cyc}
\DeclareMathOperator{\DR}{DR}

\DeclareMathOperator{\rev}{rev}

\DeclareMathOperator{\sgn}{sgn}
\DeclareMathOperator{\lex}{lex}
\DeclareMathOperator{\tSW}{\widetilde{SW}}
\DeclareMathOperator{\SW}{SW}
\DeclareMathOperator{\PP}{\mathbf{P}}

\title{Extending the Science Fiction and the Loehr--Warrington formula}

\author{Donghyun Kim}
\address{Department of Mathematical Sciences \\ Seoul National University \\Seoul 151-247 \\ Korea}
\email{hyun920310@snu.ac.kr}

\author{Jaeseong Oh}
\address{June E Huh Center for Mathematical Challenges \\ Korea Institute for Advanced Study \\ Seoul 02455 \\ Korea}
\email{jsoh@kias.re.kr}

\begin{document}

\begin{abstract}
We introduce the Macdonald piece polynomial \( \I_{\mu,\lambda,k}[X;q,t] \), which is a vast generalization of the Macdonald intersection polynomial in the science fiction conjecture by Bergeron and Garsia. We demonstrate a remarkable connection between \( \I_{\mu,\lambda,k} \), \( \nabla s_\lambda \), and the Loehr--Warrington formula \(\LW_{\lambda}\), thereby obtaining the Loehr--Warrington conjecture as a corollary. To connect \(\I_{\mu,\lambda,k}\) and \( \nabla s_\lambda \), we employ the plethystic formula for the Macdonald polynomials of Garsia--Haiman--Tesler, and to connect \(\I_{\mu,\lambda,k}\) and \(\LW_{\lambda}\), we use our new findings on the combinatorics of \(P\)-tableaux together with the column exchange rule. We also present an extension of the science fiction conjecture and the Macdonald positivity by exploiting \( \I_{\mu,\lambda,k} \).
\end{abstract}

\maketitle

\section{Introduction}\label{Sec: Intro}
In an effort to extend the science fiction conjecture by \cite{BG99}, we introduce the Macdonald piece polynomial \(\I_{\mu,\lambda,k}\), which possesses remarkable properties. To provide a comprehensive overview of the primary results, we discuss the historical context and background of Macdonald polynomials, the Loehr--Warrington conjecture, and the science fiction conjecture.

\subsection{Macdonald polynomials}
The \emph{Macdonald $P$ polynomials}, denoted as $P_\mu[X;q,t]$, are symmetric functions over $\mathbb{Q}(q,t)$ which form a basis for the ring of symmetric functions. Macdonald introduced them in 1988 \cite{Mac88} as a unified generalization of two significant families of symmetric functions, the \emph{Jack polynomials} and the \emph{Hall--Littlewood polynomials}. Since then, Macdonald polynomials have been established as fundamental objects with deep connections across various areas of mathematics. These include their appearance and relationships with affine Hecke algebras, quantum groups, elliptic Hall algebras, Hilbert scheme, affine Springer fiber, the Calogero–Sutherland model, diagonal coinvariants, and combinatorics.

An important question posed by Macdonald was the \emph{Macdonald positivity conjecture}. This conjecture asserts that the \emph{modified Macdonald polynomial}, defined through a plethystic substitution to the \emph{Macdonald $J$ polynomial} in \cite{Mac88},
\[
    \widetilde{H}_\mu[X;q,t] = t^{n(\mu)} J_\mu\left[\frac{X}{1 - t^{-1}}; q, t^{-1}\right],
\]
is Schur positive. To address this conjecture, Garsia and Haiman introduced the bigraded \(\mathfrak{S}_n\) module \( R_\mu \) in 1993, now called the \emph{Garsia--Haiman module}. This module is associated with a partition \(\mu \vdash n\) and is defined as the subspace of the polynomial ring \(\mathbb{C}[x_1, \dots, x_n, y_1, \dots, y_n]\) of $2n$ variables, spanned by the partial derivatives of the \emph{generalized Vandermonde determinant}
\[
    \Delta_\mu \coloneqq \det\left( x_k^{i-1} y_k^{j-1} \right)_{\substack{1 \le k \le n \\ (i,j) \in \mu}}
\]
associated with \(\mu\) \cite{GH93, GH96}.

In his groundbreaking work \cite{Hai01}, Haiman utilized the geometry of the Hilbert scheme of \( n \) points in the plane \( \mathbb{C}^2 \) to prove that
\begin{equation}\label{eq:Frob R_mu = H_mu}
    \mathrm{grFrob}\left( R_\mu; q, t \right) = \widetilde{H}_\mu[X; q, t],
\end{equation}
thereby settling the Macdonald positivity conjecture. Another consequence of Haiman's result is the \emph{\( n! \) theorem}, which states that
\begin{equation}\label{eq: n! theorem}
    \dim\left( R_\mu \right) = n!.
\end{equation}

Despite Haiman's resolution of Schur positivity, an explicit combinatorial formula for the Schur coefficients remains elusive.
\begin{problem}\label{prob: Macdonald-Kostka}
Find a combinatorial description for the Schur coefficient \(\widetilde{K}_{\lambda, \mu}(q, t)\) of the Macdonald polynomials:
\[
    \widetilde{H}_\mu[X;q,t] = \sum_{\lambda} \widetilde{K}_{\lambda, \mu}(q, t) s_\lambda,
\]
where \(s_\lambda\) is the Schur function.
\end{problem}
This open problem is one of the most fundamental and challenging questions in the theory of Macdonald polynomials and has been a driving force behind many major breakthroughs. One such advancement is the Haglund--Haiman--Loehr formula \cite{HHL05}, which provides an explicit monomial expansion for the Macdonald polynomials in terms of words and two statistics, \(\mathrm{inv}_\mu\) and \(\mathrm{maj}_\mu\) (see Theorem~\ref{thm: HHL formula}).

\subsection{Extending science fiction conjecture}\label{subsec: extending SF} 
One possible approach to resolving Problem~\ref{prob: Macdonald-Kostka} is to decompose \(\widetilde{H}_{\mu}\) into pieces that are Schur positive, and then find a combinatorial description of the Schur coefficients for each piece. The celebrated HHL formula breaks down \(\widetilde{H}_{\mu}\) into LLT polynomials, which are known to be Schur positive \cite{GH07}. However, describing the Schur coefficients of LLT polynomials remains a challenging open problem. Another attempt involved using \(k\)-Schur functions \cite{LLM03}, whose Schur coefficients were explicitly described in \cite{BMPS19}. However, it is still unclear whether \(\widetilde{H}_{\mu}\) can be expressed positively in terms of \(k\)-Schur functions. In this paper, we introduce an alternative way to decompose \(\widetilde{H}_{\mu}\), generalizing the way proposed in the science fiction conjecture.

Bergeron and Garsia studied the Garsia–-Haiman modules using inductive and combinatorial methods \cite{BG99}. They introduced a collection of conjectures referred to as the \emph{science fiction conjecture}, which include the existence of a certain basis for Garsia--Haiman modules and the elegant properties of those bases while intersecting Garsia--Haiman modules. To elaborate on the science fiction conjecture, let \(\mu\) be a partition with \(n\) (removable) corners. For \(1 \le i \le n\), let \(\mu^{(i)}\) denote the partition obtained from \(\mu\) by removing the \(i\)-th corner. The science fiction conjecture then implies the following two assertions, which are analogous to \eqref{eq:Frob R_mu = H_mu} and \eqref{eq: n! theorem}, respectively:

\begin{enumerate}
    \item The bigraded Frobenius characteristic of the intersection of the modules \(\bigcap_{i=1}^{k} R_{\mu^{(i)}}\) is given by
    \begin{equation}\label{eq: SF int R_mu=Mac int}
        \grFrob\left(\bigcap_{i=1}^{k} R_{\mu^{(i)}};q,t\right) = \I_{\mu^{(1)},\dots,\mu^{(k)}}[X;q,t]\coloneqq\sum_{i=1}^{k} \left(\prod_{j\neq i}\dfrac{T_{\mu^{(j)}}}{T_{\mu^{(j)}}-T_{\mu^{(i)}}}\right)\widetilde{H}_{\mu^{(i)}}[X;q,t],
    \end{equation}
    \item the dimension of the intersection \(\bigcap_{i=1}^{k} R_{\mu^{(i)}}\) is given by
    \begin{equation}\label{eq: n!/k conjecture}
        \dim\left(\bigcap_{i=1}^{k} R_{\mu^{(i)}}\right) = \frac{n!}{k}.
    \end{equation}
\end{enumerate}
Here, $T_\mu\coloneqq\prod_{(i,j)\in\mu}t^{i-1}q^{j-1}$. We refer to the symmetric function \(\I_{\mu^{(1)}, \dots, \mu^{(k)}}\) in \eqref{eq: SF int R_mu=Mac int} as the \emph{Macdonald intersection polynomial}, and the second part as the \emph{\(n!/k\) conjecture}. The Macdonald intersection polynomial naturally refines the usual Macdonald polynomials. As such, the combinatorics of Macdonald intersection polynomials may lead to progress on Problem~\ref{prob: Macdonald-Kostka}. The authors, along with Seung Jin Lee, recently provided an explicit positive monomial expansion for the Macdonald intersection polynomials for two partitions \cite{KLO22}, refining the celebrated HHL formula and providing a step toward Butler's conjecture \cite{But94}.

In this paper, we further generalize the science fiction conjecture. Especially, we generalize the Macdonald intersection polynomial in \eqref{eq: SF int R_mu=Mac int} and the $n!/k$ conjecture \eqref{eq: n!/k conjecture}. To be precise, like in the science fiction conjecture, let $\mu$ be a partition with $n$ removable corners $\{c_1,\dots,c_n\}$ and let $1\le k \le n-1$. For a subset $S \in \binom{[n]}{k}$ of $[n]$ of size $k$, let $\mu^{S}$ be the partition obtained from $\mu$ by deleting corners $\{c_i : i \in S\}$. Let $R_{n,k} := ((n-k)^k)$ be the rectangular partition, and let $\lambda$ be a partition inside $R_{n,k}$. We define the \emph{Macdonald piece polynomial} $\I_{\mu,\lambda,k}$ indexed by $\mu$, $k$, and $\lambda$ as
\begin{equation}\label{eq: def of I_mu,lambda,k}
    \I_{\mu,\lambda,k}[X;q,t] := (-1)^{|\lambda|}\sum_{S \in \binom{[n]}{k}} \dfrac{s_\lambda[z_S] \prod_{j \in S^c}{z_j}}{\prod_{\substack{i \in S \\ j \in S^c}} (z_j - z_i)} \widetilde{H}_{\mu^S}[X].    
\end{equation}
Here, we let $z_i = T_{\mu^{\{i\}}} / T_\mu = q^{-\coarm(c_i)} t^{-\coleg(c_i)}$ and $S^{c}=\{1,2,\,\dots,n\}\setminus S$. Additionally, $f[z_S]$ denotes the polynomial $f$ in the $z_i$'s where $i$ runs through $S$. For $k=1$ and the empty partition $\lambda=\emptyset$, $\I_{\mu,\lambda,k}$ recovers the Macdonald intersection polynomial $\I_{\mu^{(1)},\dots,\mu^{(n)}}$.

In Section~\ref{sec: preliminaries extending the SF}, using the novel symmetric function \( \I_{\mu,\lambda,k} \), we define another symmetric function \( \widetilde{H}_{\mu^S}^{\lambda} \) (See \eqref{eq: module intersection sym def}). We conjecture that it is the graded Frobenius characteristic of the intersection of certain Garsia--Haiman modules (see Conjecture~\ref{conj: sf extend}), largely generalizing the science fiction conjecture given in \eqref{eq: SF int R_mu=Mac int}. In the appendix~\ref{sec: Appendix}, as a counterpart for the \( n!/k \) conjecture \eqref{eq: n!/k conjecture}, numbers corresponding to each partition and their conjectural properties are outlined.

In addition to its crucial role in generalizing the science fiction conjecture, Macdonald piece polynomial \( \I_{\mu,\lambda,k} \) enjoys a remarkable property. We unveil a surprising connection with $\nabla s_{\lambda}$ and the Loehr--Warrington formula $\LW_{\lambda}$, where the Loehr--Warrington conjecture follows as a corollary. We first give a background of the Loehr--Warrington conjecture.


\subsection{The Loehr--Warrington conjecture}
The theory of Macdonald polynomials has seen significant advancement due to its deep connections with algebraic geometry, operators on symmetric functions, and Catalan combinatorics. A key result that exemplifies these connections is Haiman's \((n+1)^{n-1}\) theorem. By leveraging the geometric structure of the Hilbert scheme of \(n\) points in the plane \(\mathbb{C}^2\), Haiman proved that the dimension of the \emph{diagonal coinvariant ring},
\[
    \DR_n \coloneqq \mathbb{C}[x_1, \dots, x_n, y_1, \dots, y_n] / \langle \mathbb{C}[x_1, \dots, x_n, y_1, \dots, y_n]^{\mathfrak{S}_n}_+ \rangle,
\]
is \((n+1)^{n-1}\). Furthermore, he established that the bigraded Frobenius characteristic of this ring coincides with the symmetric function \(\nabla e_n\), where \(\nabla\) is the eigenoperator on Macdonald polynomials defined by \(\nabla \widetilde{H}_\mu = T_\mu \widetilde{H}_\mu\), and \(e_n\) is the elementary symmetric function.

The celebrated \emph{Shuffle Theorem}, conjectured in \cite{HHLRU05} and proved in \cite{CM18}, further underscores the interplay between Macdonald polynomials, operators on symmetric functions, and Catalan combinatorics. This theorem provides an explicit combinatorial expansion for \(\nabla e_n\) in terms of the combinatorics of labeled Dyck paths, showcasing the rich combinatorial structure. In an earlier work, the authors, along with Seung Jin Lee, offered a new proof of the shuffle theorem by connecting the Macdonald intersection polynomial $\I_{\mu^{(1)},\dots,\mu^{(n)}}$ to \(\nabla e_{n-1}\) \cite{KLO23}.

Over the past two decades, the shuffle theorem has been extended the rational shuffle theorem~\cite{Mel21} and the Delta theorem~\cite{DM22, BHMPS21Delta}. The Loehr--Warrington conjecture \cite{LW08}, the final outstanding conjecture in this series, was very recently proved in \cite{BHMPS21LW}, providing a combinatorial formula for \(\nabla s_\lambda\) in terms of \emph{nested Dyck paths}. We use another formulation expressing the combinatorial formula in the Loehr--Warrington conjecture in terms of certain \emph{$P$-tableaux} \cite{Ges84}:
\begin{equation*}
\LW_{\lambda}:=q^{\adj(\lambda)} \sum_{T\in \mathcal{T}(\lambda)}q^{\dinv(T)}t^{\area(T)}x^T.
\end{equation*}
See Section~\ref{subsec: LW formula} for precise definition for the set $\mathcal{T}(\lambda)$ of certain $P$-tableaux, and the statistics $\dinv$ and $\area$. It turns out the combinatorics of $P$-tableau beautifully match with the combinatorics of Macdonald piece polynomials.

\subsection{Main results}
Before stating our main results, we introduce a notation that we will use throughout this paper. For a partition $\lambda $ inside a rectangle $R(n,k)$, we define $\tilde{\lambda}$ to be a conjugate partition of $(n-k-\lambda_k,\dots,n-k-\lambda_1)$ where we regard $\lambda_i=0$ if $i>\ell(\lambda)$. The notation $\tilde{\lambda}$ is dependent on an ambient rectangle that $\lambda$ lives in, which will be obvious from the context. 

Now we state our main results. Note that the straightforward implication of Theorem \ref{thm: main theorem} (b) and (c) is the Loehr--Warrington conjecture (Corollary \ref{cor: LW thrm}).

\begin{thm}\label{thm: main theorem} Let $\mu$ be a partition with $n$ corners, and $1\le k<n$. Then for a partition $\lambda$ inside a rectangle $R(n,k)$ the symmetric function $\I_{\mu,\lambda,k}$ satisfies the following:
\begin{enumerate}[(a)]
    \item For $N>|\mu|-|\tilde{\lambda}|-k$, we have
    \begin{equation}\label{eq: vanishing identity}
        e_N^\perp \I_{\mu,\lambda,k} = 0.
    \end{equation}
    \item We have
    \begin{equation}\label{eq: e^perp I = nabla s_lambda}
        \dfrac{1}{T_{\mu^{[n]}}}e_{|\mu|-|\tilde{\lambda}|-k}^\perp \I_{\mu,\lambda,k} = \nabla s_{\tilde{\lambda}}.
    \end{equation}
    In particular, $e_{|\mu|+|\tilde{\lambda}|-k}^\perp \I_{\mu,\lambda,k}$ does not depend on the partition $\mu$ (up to a constant).
    \item We have
    \begin{equation}\label{eq: e^perp I = LW formula}
        \dfrac{1}{T_{\mu^{[n]}}}e_{|\mu|-|\tilde{\lambda}|-k}^\perp \I_{\mu,\lambda,k} = (-1)^{\adj(\tilde{\lambda})} \LW_{\tilde{\lambda}}.
    \end{equation}
\end{enumerate}
Here, $e_n^\perp$ denotes the adjoint operator to the multiplication by $e_n$ with respect to the Hall inner product. The precise definitions of the Loehr-Warrington formula $\LW_{\lambda}$ and $\adj(\lambda)$ will be given in Section 4.  
\end{thm}

\begin{corollary}\label{cor: LW thrm} (The Loehr--Warrington conjecture) 
    We have $\nabla s_{\lambda}=(-1)^{\adj(\lambda)}\LW_{\lambda}$.
\end{corollary}

We conjecture that \( (-1)^{\adj(\tilde{\lambda})} \I_{\mu,\lambda,k} \) is Schur positive. We do not know the monomial positivity either, and Theorem \ref{thm: main theorem} (c) can be thought of as supporting evidence. Finding a combinatorial description of the monomial expansion for \(\I_{\mu,\lambda,k}\) is a challenging open problem, as one would need to discover an ambient formula that includes the Loehr--Warrington formula as a small portion.

We also study the specialization of the Macdonald piece polynomials at \(q=t=1\). In particular, we describe the \(h\)-expansion in terms of \(W_{\lambda,\mu} := [h_{\mu}](\nabla s_{\lambda}\vert_{q,t=1})\), the coefficient of \(h_{\mu}\) in the \(h\)-expansion of \(\nabla s_{\lambda}\vert_{q,t=1}\). In \cite{KLO23}, the authors proved that \(W_{(1^n),\mu}\) is given by the Kreweras number, but we do not yet have a combinatorial description of \(W_{\lambda,\mu}\) in general beyond the case \(\lambda = (1^n)\). Now, we state a vast generalization of \cite[Theorem 1.3]{KLO23}.

\begin{thm}\label{thm: q=t=1}
We have 
\begin{equation*}
    \I_{\mu,\lambda,k}[X;1,1]=\sum_{\nu \subseteq \lambda}(-1)^{|\lambda|-|\nu|}d^{(k)}_{\lambda,\nu}\sum_{\tau}W_{\tilde{\nu},\tau}h_{(\tau+1^{|\mu|-k-|\tilde{v}|})}
\end{equation*}
where 
\begin{equation*}
    d^{(k)}_{\lambda,\nu}=\det\left(\binom{\lambda_i+k-i}{\nu_j+k-j}\right)_{1\leq i,j\leq k}
\end{equation*}
and $\tilde{\nu}$ is taken regarding $\nu$ is inside a rectangle $R(n,k)$. 
\end{thm}

\begin{example}
    Let $\mu=(4,3,2,1)$, $k=2$, $n=4$ and $\lambda=(2,1)$. Then we have
    \begin{equation*}
        d^{(2)}_{(2,1),(2,1)}=1,\quad d^{(2)}_{(2,1),(1,1)}=3,\quad d^{(2)}_{(2,1),(2)}=1, \quad d^{(2)}_{(2,1),(1)}=3, \quad \text{and}\quad d^{(2)}_{(2,1),()}=2.
    \end{equation*}
    From
    \begin{align*}
        &\nabla s_{(2,2)}\vert_{q,t=1}=-h_{(1,1,1,1)}+h_{(2,1,1)}+h_{(2,2)}-h_{(3,1)
        }, \quad \nabla s_{(2,1)}\vert_{q,t=1}=-3h_{(1,1,1)}+4h_{(2,1)}-h_{(3)}, \\
         &\nabla s_{(1,1)}\vert_{q,t=1}=2h_{(1,1)}-h_{(2)},\quad \nabla s_{(2)}\vert_{q,t=1}=-h_{(1,1)}+h_{(2)}, \quad \text{and} \quad  \nabla s_{(1)}\vert_{q,t=1}=h_{(1)},
        \end{align*}
      we obtain
      \begin{align*}
         \I_{\mu,\lambda,k}[X;1,1]=&(h_{(2,1,1,1,1,1,1)})-3(-h_{(2,2,1,1,1,1)}+h_{(3,1,1,1,1,1)}))-(2h_{(2,2,1,1,1,1)}-h_{(3,1,1,1,1,1)})\\
         &+3(-3h_{(2,2,2,1,1)}+4h_{(3,2,1,1,1)}-h_{(4,1,1,1,1,1)})-2(-h_{(2,2,2,2)}+h_{(3,2,2,1)}+h_{(3,3,1,1)}-h_{(4,2,1,1)
        })\\
        =&h_{(2,1,1,1,1,1,1)}+h_{(2,2,1,1,1,1)}-2h_{(3,1,1,1,1,1)}-9h_{(2,2,2,1,1)}+12h_{(3,2,1,1,1)}-3h_{(4,1,1,1,1,1)}\\
        &+2h_{(2,2,2,2)}-2h_{(3,2,2,1)}-2h_{(3,3,1,1)}+2h_{(4,2,1,1)
        }.
      \end{align*}
\end{example}

\subsection{organization} This paper is organized as follows. In Section \ref{sec: preliminaries extending the SF}, we introduce the necessary concepts in symmetric function theory and then present the extension of the science fiction conjecture (Conjecture \ref{conj: sf extend}), thereby motivating the importance of the Macdonald piece polynomial. In Section \ref{sec3: GHT}, we establish Theorem \ref{thm: main theorem} (a) and (b) utilizing the plethystic formula for Macdonald polynomials by Garsia--Tesler--Haiman. In Section \ref{sec: LW formula}, we recall the Loehr--Warrington formula given in \cite{LW08} and present its deformation in terms of the operators. In Section \ref{sec5: column exchange filled diagram}, we explore generalized Macdonald polynomials for filled diagrams and the column exchange rule introduced in \cite{KLO22}. In Section \ref{sec: deform}, we finalize the proof of Theorem \ref{thm: main theorem} by showing that the combinatorics of the Macdonald piece polynomial elegantly matches with the Loehr--Warrington formula. In Section \ref{sec: q=t=1}, we prove Theorem \ref{thm: q=t=1} on the $h$-expansion of Macdonald piece polynomials at $q=t=1$. In Section \ref{sec: future questions}, we discuss some future questions arising from this work. Finally, Appendix~\ref{Sec: Appendix} includes data with interesting numbers, ratios of dimensions.

\section*{acknowledgement}
The authors are grateful to François Bergeron, James Haglund, Seung Jin Lee and Brendon Rhoades for helpful conversations. The authors
also appreciate to Raymond Chou for sharing his SAGE code for computing Garsia--Haiman modules. D. Kim was supported by NRF grants partially supported by Science Research Center Program through the National Research Foundation of Korea(NRF) Grant funded by the Korean Government and NRF grants 2022R1I1A1A01070620. J. Oh was supported by KIAS Individual Grant (HP083401) at  June E Huh Center for Mathematical Challenges in Korea Institute for Advanced Study.

\section{Extending the Science Fiction conjecture with Macdonald piece polynomials}
\label{sec: preliminaries extending the SF}
\subsection{Preliminaries}
 A partition $\mu=(\mu_1,\dots,\mu_\ell)$ is a weakly decreasing sequence of positive integers. We say that \(\mu\) is a partition of \(n\), denoted by \(\mu \vdash n\), if \(|\mu| := \sum_{i} \mu_i = n\). We denote the length of \(\mu\), i.e., the number of its parts, by \(\ell(\mu)\). We use \emph{Young diagrams} in French notation to display a partition as
\[
    \mu=\{(i,j)\in \mathbb{Z}_{\geq 1}\times \mathbb{Z}_{\geq 1} : j\le \mu_i\},
\]
where each element of this set represents a \emph{cell} in the diagram. The \emph{conjugate partition} of $\mu$, denoted by $\mu'=(\mu'_1,\mu'_2,\dots)$, is obtained by reflecting the Young diagram of $\mu$ along the diagonal $y=x$. We denote the empty partition by $\emptyset$.

For a cell $u$ in a partition $\mu$, we define the \emph{arm} (\emph{coarm}) of $u$, denoted by $\arm_\mu(u)$ ($\coarm_\mu(u)$), as the number of cells strictly to the right (left) of $u$ in the same row. Similarly, we define the \emph{leg} (\emph{coleg}) of $u$, denoted by $\leg_\mu(u)$ ($\coleg_\mu(u)$), as the number of cells strictly above (below) $u$ in the same column. 

A \emph{diagram} is a collection of cells in the first quadrant. A partition can be considered as a left-justified diagram, where the number of cells in each row is weakly decreasing from bottom to top. We denote a cell in a diagram in the $i$-th row, and $j$-th column by $(i,j)$. We usually represent a diagram by
\[
    D = [D^{(1)}, D^{(2)},\dots],
\]
where $D^{(j)}=\{i: (i,j)\in D\}$ is the row numbers of the $j$-th column of $D$.
For a diagram \( D \), the \emph{size} $|D|$ is defined as the number of cells in \( D \).

Let $\mathbf{Sym}$ be the graded ring of symmetric functions in infinite variables $X=x_1,x_2,\dots$ over the ground field $\mathbb{C}(q,t)$. We use the notations from \cite{Mac88} for families of symmetric functions indexed by a partition: $m_\lambda$ for the \emph{monomial symmetric function}, $h_\lambda$ for the \emph{homogeneous symmetric function}, $e_\lambda$ for the \emph{elementary symmetric function}, $p_\lambda$ for the \emph{power sum symmetric function}, $s_\lambda$ for the \emph{Schur function},  $\widetilde{H}_\lambda$ for the \emph{(modified) Macdonald polynomial}. The (Hall) inner product on symmetric functions, denoted by $\langle -,-\rangle$, is defined as \( \langle s_\lambda, s_\mu \rangle = \delta_{\lambda,\mu}.\)

\subsection{Extending the science fiction conjecture and the Macdonald positivity}
We fix a partition $\mu$ with $n$ corners $\{c_1,c_2,\dots,c_n\}$ and let $\mu^{(i)}$ be a partition obtained by removing a corner $c_i$. Recall that the science fiction conjecture suggests that
\begin{equation}\label{eq: sf conjecture}
    \I_{\mu^{(1)},\dots,\mu^{(n)}}[X;q,t]:=\sum_{i=1}^{n}\left(\prod_{j\neq i}\dfrac{z_j}{z_j-z_i}\right) \widetilde{H}_{\mu^{(i)}}=\grFrob\left(\cap_{i=1}^{n}R_{\mu^{(i)}}\right),
\end{equation}
where $z_i=T_{\mu^{(i)}}/T_{\mu}$.
As mentioned before, $\I_{\mu^{(1)},\dots,\mu^{(n)}}[X;q,t]$ is a special case of the Macdonald piece polynomial $\I_{\mu,\lambda,k}$ obtained by letting $\lambda=\emptyset$ and $k=1$. Our goal is to further generalize \eqref{eq: sf conjecture} and refine the Macdonald positivity as well (Conjecture \ref{conj: sf extend}).

Now we consider partitions $\mu^{S}=\mu\setminus \{c_i: i\in S\}$ for $S\in \binom{[n]}{k}$. Note that $\I_{\mu,\lambda,k}$ can be regarded as a linear combination of $\widetilde{H}_{\mu^{S}}$. We present the converse by representing each $\widetilde{H}_{\mu^{S}}$ as a linear combination of $\I_{\mu,\lambda,k}$.

\begin{lem}\label{lem: H in terms of I}
We have 
\begin{equation*}
    \widetilde{H}_{\mu^{S}}=\sum_{\lambda\subseteq R(n,k)} \frac{s_{\tilde{\lambda}}[z_{S^{c}}]}{\prod_{j\in S^{c}}z_i}\I_{\mu,\lambda,k}
\end{equation*}
where $z_i = T_{\mu_{\{i\}}} / T_\mu = q^{-\coarm(c_i)} t^{-\coleg(c_i)}$.
\end{lem}
\begin{proof}
Consider a matrix $M$ whose rows are indexed by $\lambda\subseteq R(n,k)$ and columns are indexed by $S\in \binom{n}{k}$. Its entries are given by
\begin{equation*}
    M_{\lambda,S}=(-1)^{|\lambda|}\dfrac{s_\lambda[z_S] \prod_{j \in S^c}{z_j}}{\prod_{\substack{i \in S \\ j \in S^c}} (z_j - z_i)}.
\end{equation*}
Then let $M'$ be a matrix whose columns are indexed by $\lambda\subseteq R(n,k)$ and rows are indexed by $S\in \binom{n}{k}$. Its entries are given by
\begin{equation*}
    M'_{S,\lambda}=\frac{s_{\tilde{\lambda}}[z_{S^{c}}]}{\prod_{j\in S^{c}}z_i}.
\end{equation*}
By Lemma \ref{lem: Schur orthogonal in rectangle} we have that $M M'$ equals the identity matrix. 
\end{proof}

From now on we regard $z_1,z_2,\dots,z_n$ as indeterminates and let
\begin{equation*}
    \widetilde{H}_{\mu^{S}}[X;q,t,z]:=\sum_{\lambda\subseteq k \times (n-k)}\frac{s_{\tilde{\lambda}}[z_{S^{c}}]}{\prod_{j\in S^{c}}z_i}\I_{\mu,\lambda,k}.
\end{equation*}
Define $\Phi$ to be a specialization map on functions in $z_1,z_2,\dots,z_n$ given by letting $z_i = T_{\mu^{\{i\}}}/T_\mu = q^{-\coarm(c_i)} t^{-\coleg(c_i)}$. By Lemma \ref{lem: H in terms of I}, we trivially have $\Phi( \widetilde{H}_{\mu^{S}}[X;q,t,z])=\widetilde{H}_{\mu^{S}}$.

Let $\pi_{i,j}$\footnote{This operator is motivated by Butler's symmetric function \cite{But94} and for $j=i+1$, $\pi_{i,j}$ coincides with the \emph{Demazure operator $\pi_i$}.} be the operator acting on functions in $z_1,z_2,\dots,z_n$ defined by
\begin{equation*}
    \pi_{i,j}f(z_1,z_2,\cdots,z_n)=\frac{z_j f - z_i f\vert_{z_i \leftrightarrow z_j}}{z_j-z_i}.
\end{equation*}

Denoting $S=\{i_1<i_2<\dots<i_k\}$ and $S^{c}=\{j_1<j_2<\dots<j_{n-k}\}$, for a partition $\lambda \subseteq R(n,k)$ we define $\pi_{\lambda,S}$ to be a sequence of operators given by
\begin{equation*}
   \pi_{\lambda,S}=\prod_{(r,s)\in\lambda}\pi_{i_r,j_s}.
\end{equation*}
We may consider $\widetilde{H}_{\mu^{S}}[X;q,t,z]$ as a function in $z_1,z_2,\dots,z_n$ regarding $\I_{\mu,\lambda,k}$ as constants. In this sense the operators $\pi_{i,j}$ act on $ \widetilde{H}_{\mu^{S}}[X;q,t,z]$ by
\begin{equation*}
    \pi_{\lambda,S}( \widetilde{H}_{\mu^{S}}[X;q,t,z])=\sum_{\nu\subseteq R(n,k)} \left( \pi_{\lambda,S}\frac{s_{\tilde{\nu}}[z_{S^{c}}]}{\prod_{j\in S^{c}}z_i}\right)\I_{\mu,\nu,k}.
\end{equation*}
Finally, we define 
\begin{equation}\label{eq: module intersection sym def}
    \widetilde{H}^{\lambda}_{\mu^{S}}:=\Phi\left(\pi_{\lambda,S}( \widetilde{H}_{\mu^{S}}[X;q,t,z])\right).
\end{equation}
Now we give a conjectural module theoretic interpretation of $\widetilde{H}^{\lambda}_{\mu^{S}}$. To that end, we define a notation. 

\begin{definition}
    For $S\in\binom{[n]}{k}$ and a partition $\lambda \in R(n,k)$, denote by
    \begin{equation*}
        S=\{i_1<i_2<\dots<i_k\}, \qquad S^{c}=\{j_1<j_2<\dots<j_{n-k}\}.
    \end{equation*}
    We define $S(\lambda)$ to be a set of $S'\in\binom{[n]}{k}$ satisfying the following: Denote by
    \begin{equation*}
        S\setminus S'=\{i_{a_1}<i_{a_2}<\dots<i_{a_r}\}, \qquad S'\cap S^{c}=\{j_{b_1}>j_{b_2}>\dots>j_{b_r}\},
    \end{equation*}
    then we have $b_{s}\leq \lambda_{a_s}$ for $1\leq s\leq r$. For example in the extreme, $S(\emptyset)=\{S\}$ and $S(R(n,k))=\binom{[n]}{k}$. For example in the middle, letting $k=2$, $n=4$ and $\lambda=(2,1)$, 
    we have $S(\lambda)=\{\{1,2\},\{1,3\},\{2,3\},\{2,4\},\{3,4\}\}$. 
\end{definition}
\begin{conj}\label{conj: sf extend}
We have
\begin{equation*}
      \widetilde{H}^{\lambda}_{\mu^{S}}=\grFrob\left(\cap_{S'\in S(\lambda)}R_{\mu^{S'}}\right).
\end{equation*}
\end{conj}
Letting $\lambda$ to be a full rectangle $R(n,k)$ the right-hand side is the Frobenius characteristic of the full intersection. It is easy to check that the left-hand side $\widetilde{H}^{R(n,k)}_{\mu^{S}}$ is independent of the choice of $S\in \binom{[n]}{k}$. Moreover, setting $k=1$ and $\lambda=R(n,1)$, Conjecture \ref{conj: sf extend} simply reduces to \eqref{eq: sf conjecture}.

For partitions $\lambda^{(1)}\subset \lambda^{(2)}$, we trivially have $S(\lambda^{(1)}) \subset S(\lambda^{(2)})$. Therefore Conjecture \ref{conj: sf extend} implies that 
\begin{equation*}
    \widetilde{H}^{\lambda^{(1)}}_{\mu^{S}}- \widetilde{H}^{\lambda^{(2)}}_{\mu^{S}}
\end{equation*}
is Schur positive. Given a sequence $\emptyset=\lambda^{(1)}\subset \lambda^{(1)}\subset \dots \subset \lambda^{(\ell)}=R(n,k)$, we may obtain a sequence of symmetric functions that grows from $\widetilde{H}^{R(n,k)}_{\mu^{S}}$ to $\widetilde{H}^{\emptyset}_{\mu^{S}}=\widetilde{H}_{\mu^{S}}$ in a Schur positive sense. This refines the Macdonald positivity.

\section{Proof of Theorem~\ref{thm: main theorem} (a) and (b)}
\label{sec3: GHT}
\subsection{Technical Lemmas}

We begin by recalling several technical lemmas. To state the first lemma, we review the following notation
\begin{align*}
    &M = (1-q)(1-t), \qquad 
    B_\mu = \sum_{c\in\mu} q^{\coarm_\mu(c)} t^{\coleg_\mu(c)}, \qquad D_\mu = M B_\mu - 1, \\
    &\tilde{h}_\mu = \prod_{c\in\mu} (q^{\arm_\mu(c)} - t^{\leg_\mu(c) + 1}), \qquad \tilde{h}'_\mu = \prod_{c\in\mu} (t^{\leg_\mu(c)} - q^{\arm_\mu(c) + 1}).
\end{align*}
We will use brackets to denote plethystic substitutions. We use two different minus signs $-$ and $\epsilon$ to denote
\[
    p_k[-X]=-p_k[X], \qquad \text{ and } \qquad p_k[\epsilon X] = (-1)^k p_k[X].
\]
Furthermore, for a symmetric function \( f \), define \( \Pi'_f \) as
\[
    \Pi'_f[X; q, t] = \nabla^{-1} f[X - \epsilon],
\]
and \( \rev(f) \) as the \( q, t \)-reversal
\[
    \rev(f) = f\big|_{q \mapsto q^{-1}, t \mapsto t^{-1}}.
\]
In \cite{GT96}, the \( * \)-inner product \( \langle f, g \rangle_* \) is defined by
\[
    \langle p_\lambda, p_\mu \rangle_* = 
    \begin{cases}
        (-1)^{|\lambda| - \ell(\lambda)} z_\lambda \prod_{i=1}^{\ell(\lambda)} \left(1 - q^{\lambda_i}\right) \left(1 - t^{\lambda_i}\right) & \text{if } \mu = \lambda, \\
        0 & \text{otherwise}.
    \end{cases}
\]
The \( * \)-inner product exhibits numerous useful properties. In particular, it commutes with the \( q, t \)-reversal: for symmetric functions \( f \) and \( g \) of homogeneous degree \( m \), we have
\begin{equation}\label{eq: reversal *-inner product}
    \langle \rev(f), \rev(g) \rangle_* = (qt)^m \rev \left( \langle f, g \rangle_* \right).
\end{equation}

In addition, the modified Macdonald polynomials form an orthogonal basis with respect to the \( * \)-inner product \cite[Theorem 1.1]{GT96}:
\begin{equation}\label{Eq: orthogonality of Mac w.r.t <>*}
    \langle \widetilde{H}_\lambda, \widetilde{H}_\mu \rangle_* = 
    \begin{cases} 
        \tilde{h}_\mu \tilde{h}'_\mu & \text{if } \mu = \lambda, \\
        0 & \text{otherwise}.
    \end{cases}
\end{equation}
The following lemma \cite[Equation (3.4)]{KLO23} provides a formula for the Macdonald polynomial skewed by an elementary symmetric function. The proof of the lemma uses Garsia, Haiman, and Tesler's plethystic formula \cite[Theorem I.2]{GHT99}.
\begin{lem}\label{lem: skewing Macdonald}\cite{KLO23}
For a partition $\mu \vdash n$, we have
\[
    e^{\perp}_{n-m} \widetilde{H}_\mu = (qt)^m T_\mu \sum_{\lambda \vdash m} \rev \left( \Pi'_{\widetilde{H}_\lambda}[D_\mu; q, t] \right) \frac{T_\lambda \widetilde{H}_\lambda}{\tilde{h}_\lambda \tilde{h}'_\lambda}.
\]
\end{lem}
    
In the proof relating the Macdonald intersection polynomial \(\I_{\mu^{(1)}, \dots, \mu^{(k)}}\) and \(\nabla e_{k-1}\), a folklore lemma \cite[Lemma 3.3]{KLO23} played a key role. It turns out that this well-known lemma can be generalized to Lemma~\ref{lem: Schur orthogonal in rectangle}.

For $\lambda$ and $\mu$, we denote their coordinate-wise sum as $\lambda + \mu$ and concatenate them to form the vector $(\lambda, \mu)$. Additionally, let $\st_k$ represent the staircase $(k-1, k-2, \ldots, 1, 0)$. For $\lambda \subseteq R(n,k)$ and $\mu \subseteq R(n,n-k)$, we examine the expression $(\lambda + \st_k, \mu + \st_{n-k})$. In this case, we treat $\lambda$ (a vector of length $k$) and $\mu$ (a vector of length $n-k$) by appending zeros as necessary. It can be easily verified that $(\lambda + \st_k, \mu + \st_{n-k})$ forms a rearrangement of the staircase $\st_n$ if and only if $\mu = \tilde{\lambda}$. For instance, when $n = 6$ and $k = 3$, we find $R(n,k) = (3,3,3)$. Given $\lambda = (3,1)$, we compute $\tilde{\lambda} = (2,2,1)$, which leads to $\lambda + \st_k = (5,2,0)$ and $\tilde{\lambda} + \st_{n-k} = (4,3,1)$. Therefore, the combined vector \((\lambda + \st_k,\tilde{\lambda} + \st_{n-k}) = (5,2,0,4,3,1)\) rearranges to form the staircase $\st_n$.

\begin{lem}\label{lem: Schur orthogonal in rectangle}
Let \( z_1, \dots, z_n \) be variables. Let \( \lambda \) be a partition contained within the rectangle \( R(n,k) \). For a partition \( \mu \) of size \( |\mu| \le k(n-k) - |\lambda| \),
\begin{equation}\label{eq: Schur orthogonal}
    \sum_{S \in \binom{[n]}{k}} \frac{s_\lambda[z_S] s_{\mu}[z_{S^c}]}{\prod_{i \in S, j \in S^c} (z_j - z_i)} = (-1)^{|\lambda|}\delta_{\tilde{\lambda}, \mu}.
\end{equation}
\end{lem}

\begin{proof}
Recall Jacobi's bi-alternant formula 
\[
    s_\lambda(z_1,\dots,z_n) = \dfrac{\displaystyle a_{\lambda+\st_n}(z_{1},z_{2},\dots ,z_{n})}{\displaystyle a_{\st_n}(z_{1},z_{2},\dots ,z_{n})},
\]
for the Schur polynomials, where $a_{\lambda}(z_1,\dots,z_n)$ is the alternating polynomial defined by
\[
a_{\lambda}(z_1,\dots,z_n) = \det\left(z_j^{\lambda_i}\right)_{1\leq i,j\leq n}.
\]
In particular, $a_{\st_n}(z_1,\dots,z_n)$ is the usual Vandermonde determinant. The left-hand side of \eqref{eq: Schur orthogonal} becomes
\begin{equation}\label{eq: Vandermonde idea}
    \sum_{S \in \binom{[n]}{k}} \frac{s_\lambda[z_S] s_{\mu}[z_{S^c}]}{\prod_{i \in S, j \in S^c} (z_j - z_i)} = \sum_{S\in\binom{[n]}{k}} \dfrac{a_{\lambda+\st_k}(z_S)a_{\mu+\st_{n-k}}(z_{S^c})}{a_{\st_k}(z_S)a_{\st_{n-k}}(z_{S^c})\prod_{i\in S, j \in S^c}(z_j-z_i)} =  \dfrac{\sum_{S\in\binom{[n]}{k}} \sgn(S) a_{\lambda+\st_k}(z_S)a_{\mu+\st_{n-k}}(z_{S^c})}{a_{\st_n}(z_{[n]})}  
\end{equation}
where $\sgn(S)=(-1)^{(\sum_{r=k+1}^{n}r)-\sum_{s\in S}s}$.

The Laplace expansion of a determinant by complementary minors, we have
\[
    a_{(\lambda+\st_k,\mu+\st_{n-k})}(z_{[n]}) = \sum_{S\in\binom{[n]}{k}} \sgn(S) a_{\lambda+\st_k}(z_S)a_{\mu+\st_{n-k}}(z_{S^c}).
\]
By substituting this into \eqref{eq: Vandermonde idea}, we obtain
\[
    \sum_{S \in \binom{[n]}{k}} \frac{s_\lambda[z_S] s_{\mu}[z_{S^c}]}{\prod_{i \in S, j \in S^c} (z_j - z_i)} = \dfrac{a_{(\lambda+\st_k,\mu+\st_{n-k})}(z_{[n]})}{a_{\st_n}(z_{[n]})}.
\]
Since the numerator $a_{(\mu+\st_{n-k},\lambda+\st_k)}$ is an alternating polynomial of degree $|\lambda|+|\mu|+|\st_k|+|\st_{n-k}|\le |\st_n|$, it is nonzero if and only if the rearrangement of $(\mu+\st_{n-k},\lambda+\st_k)$ is exactly the staircase $\st_n$. In that case, we have
\[
    \dfrac{a_{(\lambda+\st_k,\mu+\st_{n-k})}(z_{[n]})}{a_{\st_n}(z_{[n]})} = (-1)^{|\lambda|}\dfrac{a_{\st_n}(z_{[n]})}{a_{\st_n}(z_{[n]})} = (-1)^{|\lambda|}.
\]
\end{proof}

\begin{rmk}
Lemma~\ref{lem: Schur orthogonal in rectangle} can be seen as an application of the Atiyah–Bott localization theorem for equivariant cohomologies \cite{AB84}. More precisely, let \(M\) be a smooth manifold that admits an action by a compact connected Lie group \(G\). If the fixed points \(M^G\) consist of isolated points \(p\) with associated normal bundles \(\nu_p\) in \(M\) with Euler classes \(e(\nu_p)\), then the localization formula asserts that
\begin{equation}\label{eq: localization formula}
\int_M \phi = \sum_{p \in M^G} \int_p \frac{i^* \phi}{e(\nu_p)},    
\end{equation}
where \(i: M^G \hookrightarrow M\). 

If we let \(M\) be the Grassmannian \(\text{Gr}(n,k)\) with the natural torus \(G = (\mathbb{C}^*)^n\) action, then the fixed points are in one-to-one correspondence with \(k\)-subsets of \([n]\). Recall that the cohomology ring of the Grassmannian has an integral basis described by Schur functions. Localizing the Schur function indexed by the full-rectangle \(R(n,k)\) in \eqref{eq: localization formula} proves Lemma~\ref{lem: Schur orthogonal in rectangle}.
\end{rmk}

\subsection{Proof of Theorem~\ref{thm: main theorem} (a) and (b)}

Let \(\mu\) be a partition of \(N\) with \(n\) corners \(\{c_1, \dots, c_n\}\). Fix \(0 \le k \le n\) and let \(\lambda\) be a partition inside the rectangular partition $R(n,k)$. Let \(m = k(n-k) - |\lambda|=|\tilde{\lambda}|\). By the definition of \(\I_{\mu,\lambda,k}\), for an integer \(\ell\),
\begin{equation}
    e^\perp_{N-k-\ell} \I_{\mu,\lambda,k} = \sum_{S \in \binom{[n]}{k}} \frac{s_\lambda[z_S] \prod_{j \in S^c} z_j}{\prod_{i \in S, j \in S^c} (z_j - z_i)} e^\perp_{N-k-\ell} \widetilde{H}_{\mu^S}.
\end{equation}
By Lemma~\ref{lem: skewing Macdonald}, we have
\begin{equation}\label{eq: skewing I}
    e^\perp_{N-k-\ell} \I_{\mu,\lambda,k} = (qt)^\ell T_{\mu^{[n]}} \sum_{\nu \vdash \ell} \sum_{S \in \binom{[n]}{k}} \frac{s_\lambda[z_S]}{\prod_{i \in S, j \in S^c} (z_j - z_i)} \rev(\Pi_{\widetilde{H}_\nu}' [D_{\mu^S}(q,t); q, t]) \frac{T_\nu \widetilde{H}_\nu}{\tilde{h}_\nu \tilde{h}_\nu'}.
\end{equation}
Here, we used \(T_{\mu^{[n]}}= T_{\mu^S} \prod_{j \in S^c} z_j\). Note that 
\[
    \rev\left(\Pi'_{\widetilde{H}_\nu}[D_{\mu^S}(q,t); q, t]\right) = T_\nu \widetilde{H}_\nu \left[D_{\mu^{[n]}}(q^{-1}, t^{-1}) - \epsilon + \frac{M}{qt} z_{S^c}; q^{-1}, t^{-1}\right]
\]
is a polynomial in \(z_{S^c}\) of degree \(|\nu| = \ell\). If, on the right-hand side of \eqref{eq: skewing I}, the degree of the polynomial in \(z_{S^c}\) is less than \(m\), or equivalently if \(\ell < m\), then by Lemma~\ref{lem: Schur orthogonal in rectangle}, it vanishes. This proves Theorem~\ref{thm: main theorem}~(a).

Proceed to prove Theorem~\ref{thm: main theorem} (b). Now let \(\ell = m\) in \eqref{eq: skewing I}. Note that the leading term of \(\rev\left(\Pi'_{\widetilde{H}_\nu}[D_{\mu^S}(q,t); q, t]\right)\) is equal to 
\begin{align*}
T_\nu \widetilde{H}_\nu \left[\frac{M}{qt} z_{S^c}; q^{-1}, t^{-1}\right] &= (qt)^{-m}\omega \widetilde{H}_\nu \left[M z_{S^c}; q, t\right],
\end{align*}
by the symmetry relation $\omega \widetilde{H}_\mu = T_\mu \rev(\widetilde{H}_\mu)$.
Again, by Lemma~\ref{lem: Schur orthogonal in rectangle}, we only need to consider this leading term. Let us rewrite the right-hand side of \eqref{eq: skewing I} as
\begin{align*}
    &T_{\mu^{[n]}} \nabla \left(\sum_{\nu \vdash \ell} \sum_{S \in \binom{[n]}{k}} \frac{s_\lambda[z_S]}{\prod_{i \in S, j \in S^c} (z_j - z_i)} 
    \frac{\widetilde{H}_\nu \omega \widetilde{H}_\nu \left[M z_{S^c}; q, t\right]}{\tilde{h}_\nu \tilde{h}_\nu'}\right)\\
    &=T_{\mu^{[n]}} \nabla \left(\sum_{\rho \vdash \ell } \sum_{\nu \vdash \ell} \sum_{S \in \binom{[n]}{k}} \frac{s_\lambda[z_S]}{\prod_{i \in S, j \in S^c} (z_j - z_i)} 
    \frac{\widetilde{K}_{\rho', \nu}(q,t) \omega \widetilde{H}_\nu \left[M z_{S^c}; q, t\right]}{\tilde{h}_\nu \tilde{h}_\nu'} s_{\rho'}\right)\\
    &=T_{\mu^{[n]}} \nabla \left(\sum_{\rho \vdash \ell } \sum_{S \in \binom{[n]}{k}} \frac{s_\lambda[z_S] s_{\rho'}[z_{S^c}]}{\prod_{i \in S, j \in S^c} (z_j - z_i)} s_{\rho'}\right).
\end{align*}
Here, in the first equation, $\widetilde{K}_{\lambda,\mu}(q,t)$ is the \emph{Macdonald--Kostka polynomial}, the Schur coefficient of Macdonald polynomials
\[
    \widetilde{H}_\mu[X;q,t] = \sum_{\lambda\vdash n} \widetilde{K}_{\lambda,\mu}(q,t)s_\lambda.
\]
In the last equation, we used \cite[(2.67)]{Hag04}:
\[
    s_\lambda = \sum_{\beta\vdash |\lambda|} \dfrac{\tilde{K}_{\lambda',\beta}(q,t) \widetilde{H}_{\beta}[MX;q,t]}{\tilde{h}_\beta \tilde{h}'_\beta}.
\] 
Now applying Lemma~\ref{lem: Schur orthogonal in rectangle} again we obtain 
\[
T_{\mu^{[n]}} \nabla \left(\sum_{\rho \vdash \ell} \sum_{S \in \binom{[n]}{k}} \frac{s_\lambda[z_S] s_{\rho'}[z_{S^c}]}{\prod_{i \in S, j \in S^c} (z_j - z_i)} s_{\rho'}\right) = T_{\mu^{[n]}} \nabla s_{\tilde{\lambda}}.
\]

\section{Loehr-Warrington formula}

\label{sec: LW formula}
In this section, we recall the Loehr–Warrington formula, denoted by $\LW_{\lambda}$, introduced in \cite{LW08}. We then present our original contribution: Jacobi-Trudi type formula for $\LW_{\lambda}$ (Proposition \ref{prop: JT formula}), expressed using the operators defined in Definition \ref{def: operator definition}. By leveraging the relations among these operators (as established in Lemma \ref{lem: multiplication by q}), we modify the Jacobi-Trudi type formula, ultimately arriving at the expression presented in \eqref{eq: lw formula final}. In the following sections, we complete the proof of Theorem \ref{thm: main theorem} (c) by connecting the Macdonald intersection polynomial to \eqref{eq: lw formula final}.

\subsection{The Loehr--Warrington formula} \label{subsec: LW formula}
Throughout this section, we fix a poset $\PP$ on $\mathbb{Z}_{\ge 0}\times \mathbb{Z}_{\ge1}$ defined as follows\footnote{Our poset $\PP$ is both $(3+1)$-free and $(2+2)$-free, which is the condition in the famous Stanley–Stembridge conjecture regarding the $e$-positivity of chromatic symmetric functions.}. For $(a,b), (c,d) \in \mathbb{Z}_{\geq 0}\times \mathbb{Z}_{\geq 1}$ we say $(a,b) \prec_{\PP} (c,d)$ in $\PP$ if and only if $a + 1 < c$ or $a + 1 = c$ and $b \ge d$. Otherwise, we write $(a,b) \nprec_{\PP} (c,d)$. We also give a total ordering on $\mathbb{Z}_{\geq 0}\times \mathbb{Z}_{\geq 1}$ by a (mixed) lexicographic ordering. Define $(a,b)<_{\lex}(c,d)$ if $a<c$, or $a=c$ and $b>d$. Note that the ordering of the second coordinate is reversed. For example, in terms of order of $\PP$, we have
$(4,3) \prec_{\PP} (5,3)$ and $(4,2) \nprec_{\PP} (5,3)$. On the other hand, in terms of lexicographic order, we have $(4,3) \prec_{\lex} (5,3)$ and $(4,3) \prec_{\lex} (4,2)$.

Consider a tuple $L=(L_1,\dots,L_r)$ such that each $L_i$ is a finite sub(multi)set of $\mathbb{Z}_{\geq0} \times \mathbb{Z}_{\geq1}$. We define the \emph{diagonal inversion} $\dinv$ of $L$ by 
\begin{align*}
    \dinv(L)&=\sum_{i<j}\sum_{\substack{(a,b)\in L_i\\ (a',b')\in L_j}}\chi((a,b)<_{\lex}(a',b'))\chi((a,b)\nprec_{\PP}(a',b'))\\
    &=\sum_{i<j}\sum_{\substack{(a,b)\in L_i\\ (a',b')\in L_j}} \left(\chi(a=a')\chi(b>b') + \chi(a+1=a')\chi(b<b') \right).
\end{align*}
For a diagram \( D \), let \( T \) be a filling of \( D \) with elements from \(\mathbb{Z}_{\geq 0} \times \mathbb{Z}_{\geq 1} \). Denote \( T(i,j) \) to be a filling in the cell \((i,j) \in D\). We use \( T(i,j)_1 \) (resp. \( T(i,j)_2 \)) to refer to the first (resp. second) entry of \( T(i,j) \). We define the area of \( T \) as:
\[
\area(T) = \sum_{(i,j) \in D} T(i,j)_1.
\]
We use the notation \(\dinv(T)\) to denote \(\dinv(L_1, \dots, L_r)\), where \( L_i \) represents the set of entries in the \( i \)-th column of \( T \).

We fix a partition $\lambda \subseteq R(n,k)=k \times (n-k)$. Let $s$ be the size of the \emph{Durfee square}, which is the maximal number such that $\lambda_s\geq s$ (if does not exist, $s=0$). We define a \textit{dinv adjustment} of $\lambda$ denoted by $\adj(\lambda)$ as
\begin{equation*}
    \adj(\lambda)=\sum_{i=1}^{s}(\lambda_i-i).
\end{equation*}
Consider a $n$-vector
\begin{equation*}
    (\lambda_1+k-1,\lambda_2+k-2,\dots,\lambda_k,k,k+1,\dots,n-1)
\end{equation*}
where we regard $\lambda_i=0$ if $i>\ell(\lambda)$, and define $v(\lambda)$ to be a vector obtained by sorting entries in the above vector in a weakly increasing order. Then we define the \textit{bottom} of $\lambda$, denoted by $\bo(\lambda)$, as an $n$-vector given by
\begin{equation*}
\bo(\lambda) = (s+1, s+2, \ldots, s+n) - v(\lambda),
\end{equation*}
where the subtraction is performed element-wise. We also define a \textit{pivot} of $\lambda$, denoted by $\piv(\lambda)$, to be a vector $(a_1,a_2,\dots,a_s)$ of length $s$ where $a_i$ is a number satisfying
\begin{equation*}
    v(\lambda)_{a_i}=v(\lambda)_{a_i+1}=\lambda_{s+1-i}+k-(s+1-i),
\end{equation*}
i.e., the indices where $v(\lambda)$ is non increasing. Lastly we associate a diagram $D(\lambda)$ to the partition $\lambda$ as
\begin{equation*}
    D(\lambda)=[[\bo(\lambda)_1,s],[\bo(\lambda)_2,s],\dots,[\bo(\lambda)_n,s]]
\end{equation*}
and define a set $\mathcal{T}(\lambda)$ consisting of filling $T$ of $D(\lambda)$ satisfying the following conditions: 
\begin{itemize}
    \item $T(i+1,j)\succ_{\PP} T(i,j)$ and $T(i,j-1)\nsucc _{\PP} T(i,j)$
    \item for each $j> k-s$ if $j\in \piv(\lambda)$ we have $T(\bo(\lambda)_j,j)_1>0$
    \item for each $j> k-s$ if $j\notin \piv(\lambda)$ we have $T(\bo(\lambda)_j,j)_1=0$.
\end{itemize}

Note that the first condition says that $T$ (after taking appropriate reflection) is a $\PP$-tableau defined in \cite{Gas96}. 
The \emph{Loehr--Warrington formula} $\LW_{\lambda}$ is defined by
\begin{equation}\label{eq: LW formula}
\LW_{\lambda}:=q^{\adj(\lambda)} \sum_{T\in \mathcal{T}(\lambda)}q^{\dinv(T)}t^{\area(T)}x^T,
\end{equation}
where $x^T=\prod x_{T(i,j)_2}$.
\begin{example}\label{ex: Dlambda example}
Let $k=2$ and $n=5$, and consider $\lambda = (3, 2)$ within a $2 \times 3$ rectangle. To compute $\adj(\lambda)$, we first observe that $\lambda_2 \ge 2$, and $\lambda_3 < 3$. Thus, the $q$-adjustment $\adj(\lambda)$ is calculated as $(\lambda_1 - 1) + (\lambda_2 - 2) = 2$. The vector $v(\lambda)$ is the weakly increasing rearrangement $(2, 2, 3, 4, 4)$ of $(\lambda_1 + 1, \lambda_2 + 0, k + 1, \ldots, n) = (3 + 1, 2 + 0, 2, 3, 4)$. The bottom of $\lambda$ is given by $\bo(\lambda) = (3, 4, 5, 6, 7) - (2, 2, 3, 4, 4) = (1, 2, 2, 2, 3)$. The pivots are 1 and 4 since $v(\lambda)_1 = v(\lambda)_2$ and $v(\lambda)_4 = v(\lambda)_5$. Finally, $D(\lambda) = [[1, 2], [2, 2], [2, 2], [2, 2], [3, 2]] = [\{1, 2\}, \{2\}, \{2\}, \{2\}, \emptyset]$.
\end{example}

\begin{rmk}\label{remark: LW original to ours}
 In \cite{LW08}, the Loehr-Warrington formula \(\LW_{\lambda}\) is described in terms of nested labeled Dyck paths. We demonstrate how \(T \in \mathcal{T}(\lambda)\) can be naturally associated with such nested labeled Dyck paths. Specifically, we read the entries of each row of \(T\) from right to left and construct a Dyck path that satisfies the following condition: the number of boxes between the \(i\)-th upstep and the diagonal \(y = x\) equals the first coordinate of the \(i\)-th entry (counted from the right) in \(T\), while the second coordinate corresponds to its label. Finally, the starting point for each Dyck path is determined as follows: the Dyck path corresponding to the \(i\)-th row of \(T\) (counting from the bottom) begins at \((a, a)\), where \(a = \piv(\lambda)_i - i\). Refer to Figure \ref{fig:enter-label} for an example where $\lambda=(3,2)$.
\end{rmk}

\begin{figure}
    \centering
    \raisebox{50 pt}{\ytableausetup{boxsize=2.5em}
    \begin{ytableau}
     \textcolor{red}{(1,b_4)} &  \textcolor{red}{(2,b_3)} & \textcolor{red}{(1,b_2)} & \textcolor{red}{(0,b_1)} \\
     \textcolor{blue}{(0,b_5)}
     \end{ytableau}}
    \qquad
    \begin{tikzpicture}[scale=.9]
    \draw (0,0)--(0,4)--(4,4);
    \foreach \i in {0,...,3}{
    \draw (0,\i)--(4,\i);
    \draw (\i+1,4)--(\i+1,0);
    }
    \draw[line width=1.5pt, color=red] (0,0)--(0,3) (0,3)--(2,3) (2,3)--(2,4) (2,4)--(4,4);
    \draw[line width=1.5pt, color=blue] (2,2)--(2,3)
    (2,3)--(3,3);
    \draw[line width=.5pt, dashed] (0,0)--(4,4);
    \node () at (-.3, 0.5) {\textcolor{red}{$b_1$}};
    \node () at (-.3, 1.5) {\textcolor{red}{$b_2$}};
    \node () at (-.3, 2.5) {\textcolor{red}{$b_3$}};
    \draw  node[fill,circle,scale=0.5,color=red] at (0,0) {};
    \draw  node[fill,circle,scale=0.5,color=black] at (1,1) {};
    \draw  node[fill,circle,scale=0.5,color=blue] at (2,2) {};
    \node () at (1.7, 3.5) {\textcolor{red}{$b_4$}};
    \node () at (1.7, 2.5) {\textcolor{blue}{$b_5$}};
    \end{tikzpicture}    

    \caption{From $T\in \mathcal{T}(\lambda)$ to the nested labeled Dyck paths. }
    \label{fig:enter-label}
\end{figure}
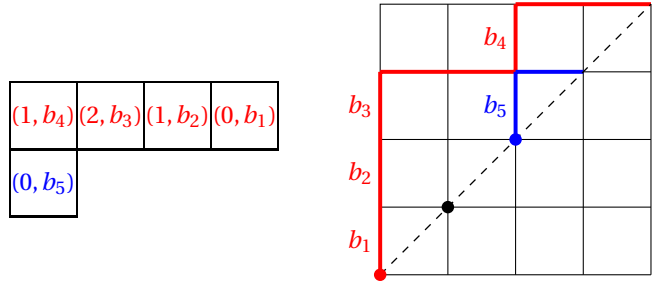

\subsection{Jacobi--Trudi type formula for the Loehr--Warrington formula} 
We provide a Jacobi--Trudi type formula for $\LW_{\lambda}$ in terms of the operators $\mathfrak{h}_m, \bar{\mathfrak{h}}_m$, and $\hat{\mathfrak{h}}_m$, defined below.
\begin{definition}\label{def: operator definition}
We define $\mathbf{C}_m$ to be a set of all $\PP$-chains $\{(a_1,b_1) \prec_{\PP} (a_2,b_2)\prec_{\PP} \dots \prec_{\PP} (a_m,b_m)\}$ where each $(a_\ell,b_\ell) \in \mathbb{Z}_{\geq0} \times \mathbb{Z}_{\geq1}$. Then $\bar{\mathbf{C}}_m$ (resp $\hat{\mathbf{C}}_m$) is a subset of $\mathbf{C}_m$ consisting of elements satisfying $a_1=0$ (resp $a_1>0$). Obviously, $\mathbf{C}_m=\bar{\mathbf{C}}_m \cup\mkern-11.5mu\cdot\mkern5mu \hspace{0.05cm}\hat{\mathbf{C}}_m$.

Let $\textbf{y}=\{y_{i,j}\}_{i \in \mathbb{Z}_{\geq 0}, j \in \mathbb{Z}_{\geq 1}}$ be a set of indeterminates . For a multiset $A$ of elements in $\mathbb{Z}_{\geq 0} \times \mathbb{Z}_{\geq 1}$, let $y^A$ be the monomial whose exponent of $y_{i,j}$ is given by the number of elements $(i,j)$ in $A$. We define the operators $\mathfrak{h}_m$, $\bar{\mathfrak{h}}_m$, and $\hat{\mathfrak{h}}_m$ acting on a polynomial ring $\mathbb{F}[\textbf{y}]$ ($\mathbb{F}$ is a ground field containing $\mathbb{C}(q)$) by describing their actions on a monomial as follows:
\begin{align*}
\mathfrak{h}_m \cdot y^A = \sum_{L \in \mathbf{C}_m} q^{\dinv(L,A)} y^{(L,A)}, \quad \
\bar{\mathfrak{h}}_m \cdot y^A = \sum_{L \in \bar{\mathbf{C}}_m} q^{\dinv(L,A)} y^{(L,A)}, \quad \
\hat{\mathfrak{h}}_m \cdot y^A = \sum_{L \in \hat{\mathbf{C}}_m} q^{\dinv(L,A)} y^{(L,A)},
\end{align*}
where $y^{(L,A)}=y^L y^A$. We then extend this linearly to define operators $\mathfrak{h}_n, \bar{h}_n$, and $\hat{\mathfrak{h}}_n$. If $m<0$ then they are all zero operators. 
\end{definition}

For example, we have
\begin{equation*}
     \hat{\mathfrak{h}}_2 \bar{\mathfrak{h}}_3 \mathfrak{h}_4 \cdot 1=\sum_{L\in \hat{\mathbf{C}}_2\times\bar{\mathbf{C}}_3 \times \mathbf{C}_4} q^{\dinv(L)}y^L
\end{equation*}
where $y^L=y^{L_1}y^{L_2}y^{L_3}$ for $L=(L_1,L_2,L_3)$.

We say two operators are the same if they act on $\mathbb{F}[\textbf{y}]$ in the same way. For example, $\mathfrak{h}_m = \hat{\mathfrak{h}}_m + \bar{\mathfrak{h}}_m$, which is trivial from $\mathbf{C}_m=\bar{\mathbf{C}}_m \cup\mkern-11.5mu\cdot\mkern5mu \hspace{0.05cm}\hat{\mathbf{C}}_m$.

The operators \( \mathfrak{h}_n \), \(\bar{\mathfrak{h}}_n\), and \(\hat{\mathfrak{h}}_n\) do not generally commute (for example, see Lemma \ref{lem: multiplication by q}). Nevertheless, operators of the same type do commute (Corollary \ref{cor: commute}). The following lemma, which is a step for proving Corollary \ref{cor: commute}, is inspired by the symmetry of chromatic quasisymmetric functions. Though the proof parallels \cite[Lemma 4.4, Theorem 4.5]{SW16} using an involution on proper colorings (or \( P \)-tabloids), we provide a proof for the sake of completeness.
\begin{lem}\label{lem: SW involution} For nonnegative integers \( n, m \), there exists a bijection 
\[
\SW: \mathbf{C}_n \times \mathbf{C}_m \rightarrow \mathbf{C}_m \times \mathbf{C}_n
\]
such that for $L\in \mathbf{C}_n \times \mathbf{C}_m$, $y^L=y^{\SW(L)}$ and $\dinv(L)=\dinv(\SW(L))$.
\end{lem}
\begin{proof}
Let \(L=(L_1, L_2) \in \mathbf{C}_n \times \mathbf{C}_m \). Consider a directed graph \( G(L) \) on the elements of \( L_1 \cup L_2 \) (with possible repeated entries), where edges are defined as \((a, b) \rightarrow (c, d)\) for pairs satisfying \((a, b) <_{\text{lex}} (c, d) <_{\text{lex}} (a+1, b)\). Or equivalently, when \((a, b) <_{\text{lex}} (c, d)\) and \( (a,b) \nprec_{\PP} (c, d)\). It is straightforward to see that \( G(L) \) is bipartite since elements in \( L_1 \) are not connected to each other, and the same applies to elements in \( L_2 \). 

We claim that every connected component of \( G(L) \) is a directed path of the form \((a_1, b_1) <_{\text{lex}} (a_2, b_2) <_{\text{lex}} \cdots <_{\text{lex}} (a_j, b_j)\), where the membership of \((a_i, b_i)\) alternates between \( L_1 \) and \( L_2 \). To prove the claim, we first show that each vertex has an out-degree at most one. Suppose not. Without loss of generality, assume \((a, b) \in L_1\) has out-degree more than one. Then there exist \((c, d), (c', d') \in L_2\) such that \((a, b) \rightarrow (c, d)\) and \((a, b) \rightarrow (c', d')\). Therefore, we have \((a, b) <_{\text{lex}} (c, d) <_{\text{lex}} (a+1,b)\) and \((a, b) <_{\text{lex}} (c', d')<_{\text{lex}} (a+1,b)\) This contradicts the fact that \((c, d)\in L_2\) and \((c', d')\in L_2\) are comparable in $\PP$. Similarly, each vertex has an in-degree at most one. Finally, there cannot be a directed cycle because every directed edge is increasing in terms of the \( <_{\text{lex}} \) ordering. Combining all these proves the claim.

For a connected component (path) \( P \) of \( G(L) \), let 
\[
    \SW(P) := \begin{cases}
    (P_1, P_2) \quad \text{ if } P \text{ has an odd number of vertices,}\\
    (P_2, P_1) \quad \text{ if } P \text{ has an even number of vertices},
    \end{cases}
\]
where \( P_1 = P \cap L_1 \) and \( P_2 = P \cap L_2 \). Then define
\(
    \SW(L) := (\cup_{P} \SW(P)_1, \cup_{P} \SW(P)_2), 
\)
where the union runs over all connected components of \( G(L) \). The map $\SW$ does not change the elements in $L_1\cup L_2$, thus $y^L=y^{\SW(L)}$. It is clear to see that \( \SW(L) \in \mathbf{C}_m \times \mathbf{C}_n \) and that $\SW$ is involutive. Finally, note that \( \dinv \) counts the number of directed edges from an element in \( L_1 \) to \( L_2 \). For each connected component \( P \) in \( G(L) \), this number is preserved by \( \SW \). Thus total \( \dinv \) statistic of $L$ is preserved under the involution \( \SW \).
\end{proof}

\begin{corollary}\label{cor: commute}
    We have $\mathfrak{h}_n \mathfrak{h}_m=\mathfrak{h}_m \mathfrak{h}_n$, $\bar{\mathfrak{h}}_n \bar{\mathfrak{h}}_m=\bar{\mathfrak{h}}_m \bar{\mathfrak{h}}_n$, and $\hat{\mathfrak{h}}_n \hat{\mathfrak{h}}_m=\hat{\mathfrak{h}}_m \hat{\mathfrak{h}}_n$. 
\end{corollary}
\begin{proof}
The map $\SW$ in Lemma \ref{lem: SW involution} directly explains that $\mathfrak{h}_n \mathfrak{h}_m=\mathfrak{h}_m \mathfrak{h}_n$ and it is easy to check that $\SW$ induces bijections $\bar{\mathbf{C}}_n \times \bar{\mathbf{C}}_m\rightarrow \bar{\mathbf{C}}_m \times \bar{\mathbf{C}}_n$ and $\hat{\mathbf{C}}_n \times \hat{\mathbf{C}}_m\rightarrow \hat{\mathbf{C}}_m \times \hat{\mathbf{C}}_n$. We conclude $\bar{\mathfrak{h}}_n \bar{\mathfrak{h}}_m=\bar{\mathfrak{h}}_m \bar{\mathfrak{h}}_n$ and $\hat{\mathfrak{h}}_n \hat{\mathfrak{h}}_m=\hat{\mathfrak{h}}_m \hat{\mathfrak{h}}_n$.
\end{proof}

Consider an $n$ by $n$ square matrix $W=(W_{i,j})$ whose entry $W_{i,j}$ is an operator acting on $\mathbb{F}[\textbf{y}]$. We let $\det(W)$ to be the operator defined by
\begin{equation*}
    \det(W)=\sum_{\sigma\in \mathfrak{S}_n}(-1)^{\sgn(\sigma)}W_{\sigma(1),1}W_{\sigma(2),2}\dots W_{\sigma(n),n}. 
\end{equation*}
For example, we have $\det\begin{pmatrix}
\mathfrak{h}_2 & \mathfrak{h}_3 \\
\mathfrak{h}_4 & \mathfrak{h}_5 
\end{pmatrix}=\mathfrak{h}_2 \mathfrak{h}_5-\mathfrak{h}_4 \mathfrak{h}_3$. Upon this notation, we provide a Jacobi--Trudi type formula for the Loehr--Warrington formula $\LW_{\lambda}$.

\begin{proposition}[Jacobi-Trudi type formula]\label{prop: JT formula}
For integers $k<n$ and a partition $\lambda \subseteq R(n,k)$ we associate an $n$ by $n$ square matrix $W(\lambda)$
\begin{align*}
    W(\lambda)_{i,j}=\begin{cases*}
        \mathfrak{h}_{v(\lambda)_j-i+1} \qquad \text{if $j\le k-s$}\\
        \bar{\mathfrak{h}}_{v(\lambda)_j-i+1} \qquad \text{if $j> k-s$ and $j$ is an entry of $\piv(\lambda)$}\\
        \hat{\mathfrak{h}}_{v(\lambda)_j-i+1} \qquad \text{if $j> k-s$ and $j$ is not an entry of $\piv(\lambda)$},
    \end{cases*}
\end{align*}
where $s$ is the size of the Durfee square of $\lambda$. Then we have 
\begin{equation*}
    q^{-\adj(\lambda)} \LW_\lambda = \det(W(\lambda)) \cdot 1 \vert_{y_{i,j}=t^i x_j}
\end{equation*}
\end{proposition}

\begin{proof}
For a permutation \(\sigma \in \mathfrak{S}_n\), let \(D(\sigma)\) be the diagram 
\[
    D(\sigma) = [[\bo(\lambda)_1, s+1-\sigma(1)], [\bo(\lambda)_2, s+2-\sigma(2)], \dots, [\bo(\lambda)_n, s+n-\sigma(n)]].
\]
Note that \(D(\text{id}) = D(\lambda)\). Let \(C_\sigma\) be the set of fillings \(A\) of $D(\sigma)$ with elements in $\mathbb{Z}_{\ge 0}\times \mathbb{Z}_{\ge1}$ such that
\begin{itemize}
    \item \(A(i+1, j) \succ_{\PP} A(i, j)\),
    \item for each \(j > k-s\), if \(j \in \piv(\lambda)\), we have \(A(\bo(\lambda)_j, j)_1 > 0\),
    \item for each \(j > k-s\), if \(j \notin \piv(\lambda)\), we have \(A(\bo(\lambda)_j, j)_1 = 0\).
\end{itemize}
If \(s+i-\sigma(i)<\bo(\lambda)_i\) for some $i$, we consider that there is no filling of $D(\sigma)$.

By definition, it is straightforward to see that
\begin{equation}\label{eq: det = P-tabloid}
    \det(W(\lambda)) \cdot 1 = \sum_{\sigma \in \mathfrak{S}_n} (-1)^{\sgn(\sigma)} \prod_{i=1}^n W(\lambda)_{\sigma(i), i} \cdot 1 = \sum_{\sigma \in \mathfrak{S}_n} (-1)^{\sgn(\sigma)} \sum_{A \in C_\sigma} q^{\dinv(A)} y^A.
\end{equation}

We construct a sign-reversing involution \(\widetilde{\SW}\) on the set
\[
    B_\lambda := \{(A, \sigma) : \sigma \in \mathfrak{S}_n, A \in C_\sigma\} \setminus \{(T,\text{id}): T\in \mathcal{T}(\lambda)\}
\]
such that
\[
    \sgn(\sigma) = -\sgn(\sigma'), \qquad y^A = y^{A'}, \qquad \text{and} \qquad \dinv(A) = \dinv(A'),
\]
where we denoted \((A',\sigma') = \widetilde{\SW}(A,\sigma)\). Via such involution, all terms in \eqref{eq: det = P-tabloid} except for the ones from $\{(T,\text{id}): T\in \mathcal{T}(\lambda)\}$ cancels. This implies that we have 
\[
    \det(W(\lambda)) \cdot 1 = \sum_{T\in \mathcal{T}(\lambda)}q^{\dinv(T)}y^T,
\]
which completes the proof.

We proceed to construct \(\widetilde{\SW}\). For a given \((A, \sigma) \in B_\lambda\), a cell \((i, j)\) in \(D(\sigma)\) is called \emph{bad} in \(A\) if either of the following conditions holds:
\begin{itemize}
    \item \((i, j-1)\) is empty with \(i \geq \bo(\lambda)_{j-1}\),
    \item \(A(i, j-1) \succ_{\PP} A(i, j)\).
\end{itemize}

For \((A, \sigma) \in B_\lambda\), there must be at least one bad cell. Let \(r = r(A)\) be the smallest \(i\) such that there is a bad cell \((i, j)\). Once \(r\) is determined, let \(c\) be the largest \(j\) such that the cell \((r, j)\) is bad. Define the sets
\[
S_{c-1}(A) = \{A(i, c-1) : r \leq i \leq s + (c-1) - \sigma(c-1)\}, \quad \text{and} \quad 
S_{c}(A) = \{A(i, c) : r < i \leq s + c - \sigma(c)\}.
\]

We then define \(\widetilde{\SW}(A, \sigma) := (A', \sigma(c-1, c))\), where \(A'\) is obtained from \(A\) by applying the map \(\SW\) (as defined in Lemma~\ref{lem: SW involution}) to \((S_{c-1}(A), S_{c}(A))\) while leaving other parts unchanged, and \((c-1, c)\) is the transposition swapping \(c-1\) and \(c\). We show that $A'$ is in $C_{\sigma'}$ where $\sigma'=\sigma(c-1, c)$. It suffices to show $A'(r-1,c-1)\prec_{\PP}A'(r,c-1)$ and $A'(r,c)\prec_{\PP}A'(r+1,c)$ so that each column is still \(\PP\)-chain. Note that $A'(r,c)=A(r,c)$ and $A'(r+1,c)$ equals $A(r+1,c)$ or $A(r,c-1)$, implying $A'(r,c)\prec_{\PP}A'(r+1,c)$. If  \(A(r-1,c-1) \nprec_{\PP} A(r+1,c)\) , together with $A(r-1, c) \nprec_{\PP} A(r-1,c-1)$ from the minimality of $r$, we have
\begin{align*}
&A(r-1, c) \prec_{\PP} A(r,c) \prec_{\PP} A(r+1, c)\\
&A(r-1, c) \nprec_{\PP} A(r-1,c-1) \nprec_{\PP} A(r+1, c)
\end{align*}
which is impossible. Since  $A'(r-1,c-1)=A(r-1,c-1)$ and $A'(r,c-1)$ equals $A(r,c-1)$ or $A(r+1,c)$, we conclude $A'(r-1,c-1)\prec_{\PP}A'(r,c-1)$. 

It is straightforward to see that \((r, c)\) remains a bad cell in $A'$,  whose minimality is trivial. This guarantees that the map \(\widetilde{\SW}\) is indeed an involution. Moreover, it is evident that the map \(\widetilde{\SW}\) is sign-reversing and preserves \(y^A\) and \(\dinv(A)\).

\end{proof}

\subsection{Reformulation of Loehr--Warrington formula}
We reformulate the Loehr--Warrington formula \(\LW_{\lambda}\) from Proposition \ref{prop: JT formula}. Our main tool is Lemma \ref{lem: multiplication by q}, whose proof will be provided in Section~\ref{subsec: proof of Lemma 4.10}.  For a matrix \( A \), let \( T_{ij}(A) \) be the matrix obtained by moving the \( j \)-th column to the position of the \( i \)-th column and shifting the \( \ell \)-th column to the \( (\ell+1) \)-th position for \( i < \ell < j \). For example, we have
\begin{equation*}
A = \begin{pmatrix}
    \vertbar & \vertbar &   \vertbar &  \vertbar  & \vertbar \\
    c_{1}    & c_{2}    &    c_3    & c_{4}  & c_5  \\
    \vertbar & \vertbar &    \vertbar &  \vertbar &\vertbar
  \end{pmatrix}
\longrightarrow T_{25} A = \begin{pmatrix}
    \vertbar & \vertbar &   \vertbar &  \vertbar  & \vertbar \\
    c_{1}    & c_{5}    &    c_2    & c_{3}  & c_4  \\
    \vertbar & \vertbar &    \vertbar &  \vertbar &\vertbar
  \end{pmatrix}.
\end{equation*}

\begin{lem}\label{lem: multiplication by q}
For an integer vector \( v = (v_1, v_2, \dots, v_n) \) of length \( n \geq 2 \), consider an \( n \times n \) matrix of operators \( V = (V_{i,j}) \) given by
\[
V_{i,j} = \begin{cases*}
\hat{\mathfrak{h}}_{v_i + j - 1} \quad \text{if \( j < n \)} \\
\bar{\mathfrak{h}}_{v_i + n - 1} \quad \text{if \( j = n \)}
\end{cases*}
\]
Then we have \( (-q)^{n-1} \det V = \det T_{1,n}(V) \).
\end{lem}

Now we modify the matrix \( W(\lambda) \) given in Proposition \ref{prop: JT formula}. Let \( s \) be the maximal number such that \(\lambda_s \geq s\). Set \( W^{(0)} = W(\lambda) \) and then construct \( W^{(1)}, \dots, W^{(s)} \) recursively as follows:
\begin{itemize}
    \item (Step 1) Given \( W^{(i)} \), let \( V^{(i+1)} = T_{ab}(W^{(i)}) \) where $a=k-s+i$ and $b=\piv(\lambda)_i$.
    \item (Step 2) Add the $(\piv(\lambda)_i+1)$-th column to the \((k-s+i)\)-th column in $V^{(i+1)}$ to obtain $W^{(i+1)}$.
\end{itemize}
By Lemma \ref{lem: multiplication by q}, after performing Step 1, the determinant multiplies by \( (-q)^{\piv(\lambda)_i - (k-s+i)} = (-q)^{\lambda_{s-i+1} - (s-i+1)} \). The property of a determinant together with Corollary \ref{cor: commute} implies that Step 2 does not change the determinant. We conclude that
\begin{equation}
    (-q)^{\adj(\lambda)}\det(W(\lambda))=\det(W^{(s)}).
\end{equation}
It is easy to check that \( W^{(s)} \) is given by
\begin{equation}\label{eq: Wr matrix formula}
    W^{(s)}_{i,j} = \begin{cases*}
        \mathfrak{h}_{\lambda_{k+1-j} + j - i} \quad \text{if \( j \leq k \)} \\
        \hat{\mathfrak{h}}_{j - i} \quad \text{if \( j > k \)}.
    \end{cases*}
\end{equation}
Finally, by Proposition~\ref{prop: JT formula}, we conclude
\begin{equation}\label{eq: lw formula final}
    \LW_{\lambda} = (-1)^{\adj(\lambda)}\det(W^{(s)}) \cdot 1 \bigg|_{y_{i,j} = t^i x_j}.
\end{equation}
\begin{example}
Let $\lambda=(3,2)$ as in Example \ref{ex: Dlambda example}. Recall that $v(\lambda)=(2,2,3,4,4)$. Then $W^{(0)}=W(\lambda)$, $W^{(1)}$ and $W^{(2)}$ can be computed as below:
\begin{align*}
    W^{(0)}=\begin{pmatrix}
        \bar{\mathfrak{h}}_2 & \hat{\mathfrak{h}}_2 & \hat{\mathfrak{h}}_3 & \bar{\mathfrak{h}}_4 & \hat{\mathfrak{h}}_4\\
        \vdots&\vdots&\vdots&\vdots&\vdots\end{pmatrix}  \xrightarrow[\text{Step1}]{\text{}}\begin{pmatrix}
        \bar{\mathfrak{h}}_2 & \hat{\mathfrak{h}}_2 & \hat{\mathfrak{h}}_3 & \bar{\mathfrak{h}}_4 & \hat{\mathfrak{h}}_4\\
        \vdots&\vdots&\vdots&\vdots&\vdots \end{pmatrix} \xrightarrow[\text{Step2}]{\text{}}  \begin{pmatrix}
        \mathfrak{h}_2 & \hat{\mathfrak{h}}_2 & \hat{\mathfrak{h}}_3 & \bar{\mathfrak{h}}_4 & \hat{\mathfrak{h}}_4\\
        \vdots&\vdots&\vdots&\vdots&\vdots   
    \end{pmatrix}=W^{(1)},\\
        W^{(1)}=\begin{pmatrix}
        \mathfrak{h}_2 & \hat{\mathfrak{h}}_2 & \hat{\mathfrak{h}}_3 & \bar{\mathfrak{h}}_4 & \hat{\mathfrak{h}}_4\\
        \vdots&\vdots&\vdots&\vdots&\vdots\end{pmatrix} \xrightarrow[\text{Step1}]{\text{}} \begin{pmatrix}
        \mathfrak{h}_2 & \bar{\mathfrak{h}}_4& \hat{\mathfrak{h}}_2 & \hat{\mathfrak{h}}_3  & \hat{\mathfrak{h}}_4\\
        \vdots&\vdots&\vdots&\vdots&\vdots \end{pmatrix}  \xrightarrow[\text{Step2}]{\text{}} \begin{pmatrix}
        \mathfrak{h}_2 & \mathfrak{h}_4& \hat{\mathfrak{h}}_2 & \hat{\mathfrak{h}}_3  & \hat{\mathfrak{h}}_4\\
        \vdots&\vdots&\vdots&\vdots&\vdots \end{pmatrix}=W^{(2)}.
\end{align*}
\end{example}

\subsection{Proof of Lemma \ref{lem: multiplication by q}}\label{subsec: proof of Lemma 4.10}
In this section we prove Lemma \ref{lem: multiplication by q}. In particular, we prove Lemma \ref{lem: q multi middle} which naturally implies Lemma \ref{lem: multiplication by q}. 

\begin{lem}\label{lem: q multi middle}
For an integer vector $v=(v_1,v_2,\dots,v_n)$ of length $n\geq 2$, consider $n$ by $n$ matrices $V$ and $V'$ given as
\begin{align*}
    V_{i,j}=\begin{cases*}\hat{\mathfrak{h}}_{v_i}\qquad \text{if $j=1$}\\
    \bar{\mathfrak{h}}_{v_i+n-1}\qquad \text{if $j=2$}
    \end{cases*}, \qquad V'_{i,j}=\begin{cases*}\bar{\mathfrak{h}}_{v_i+n-1}\qquad \text{if $j=1$}\\\hat{\mathfrak{h}}_{v_i}\qquad \text{if $j=2$}.
    \end{cases*}
\end{align*}
and  $V_{i,j}=V'_{i,j}=\mathfrak{h}_{v_i+j-2}$ for $j>2$. Then we have $q \det(V)+\det(V')=0$.
\end{lem}
\begin{proof}
We proceed with induction on \( n \). For the base case \( n=2 \), we prove the claim by imitating the argument in Lemma \ref{lem: SW involution}. Let \( v = (v_1, v_2) = (a, b) \). Then, we need to prove
\begin{equation}\label{eq: Lemma 4.12 n=2}
    q \det(V) + \det(V') = q(\hat{\mathfrak{h}}_a \bar{\mathfrak{h}}_{b+1} - \hat{\mathfrak{h}}_b \bar{\mathfrak{h}}_{a+1}) + (\bar{\mathfrak{h}}_{a+1} \hat{\mathfrak{h}}_b - \bar{\mathfrak{h}}_{b+1} \hat{\mathfrak{h}}_q) = 0,
    \end{equation}
which is equivalent to 
\begin{equation}\label{eq: Lemma 4.12 n=2 modified}
    q \hat{\mathfrak{h}}_a \bar{\mathfrak{h}}_{b+1} + \bar{\mathfrak{h}}_{a+1} \hat{\mathfrak{h}}_b = q \hat{\mathfrak{h}}_b \bar{\mathfrak{h}}_{a+1} + \bar{\mathfrak{h}}_{b+1} \hat{\mathfrak{h}}_q.
\end{equation}
To that end, we define a bijection 
\[
    \tSW: \hat{\mathbf{C}}_a \times \bar{\mathbf{C}}_{b+1} \cup \bar{\mathbf{C}}_{a+1} \times \hat{\mathbf{C}}_b \rightarrow \hat{\mathbf{C}}_b \times \bar{\mathbf{C}}_{a+1} \cup \bar{\mathbf{C}}_{b+1} \times \hat{\mathbf{C}}_a
\]
as follows. Let \( L = (L_1, L_2) \in \hat{\mathbf{C}}_a \times \bar{\mathbf{C}}_{b+1} \cup \bar{\mathbf{C}}_{a+1} \times \hat{\mathbf{C}}_b \). Consider a directed graph on the elements of \( L \), where edges are defined as \( (a, b) \rightarrow (c, d) \) for those satisfying \( (a, b) <_\text{lex} (c, d) \) and \( \dinv(\{(a, b)\}, \{(c, d)\}) = 1 \). As we have seen in Lemma~\ref{lem: SW involution}, every connected component of this directed graph forms a directed path. Note that in \( L \) there is only one element of the form \( (0, *) \). Let \( P = (P_1, P_2) \) be the directed path containing this element, where \( P_1 = L_1 \cap P \) and \( P_2 = L_2 \cap P \). Then, let \( (L'_1, L'_2) = \SW(L_1 \setminus P_1, L_2 \setminus P_2) \). Then, \( \tSW(L) \) is defined by
\[
    \tSW(L) = \begin{cases}
        (P_1 \cup L'_1, P_2 \cup L'_2) & \text{if $P$ is an odd path}, \\
        (P_2 \cup L'_1, P_1 \cup L'_2) & \text{if $P$ is an even path}.
    \end{cases}
\]

Since the involution \( \SW \) preserves the \( \dinv \) statistic, and swapping the even path changes an additional \( \pm 1 \) to the \( \dinv \), it is straightforward to see that
\[
    \dinv(L) + \chi(L \in \hat{\mathbf{C}}_a \times \bar{\mathbf{C}}_{b+1}) = \dinv (\tSW(L)) + \chi(\tSW(L) \in \hat{\mathbf{C}}_b \times \bar{\mathbf{C}}_{a+1}).
\]
This proves \eqref{eq: Lemma 4.12 n=2 modified}, thus \eqref{eq: Lemma 4.12 n=2} follows.

Now let $n\geq 3$ and assume the claim is true for any number smaller than $n$. We define $n$ by $n$ matrices $W$, $W'$, $Z$ and $Z'$ as 
\begin{align*}
    W_{i,1}=\bar{\mathfrak{h}}_{v_i+1},\quad W_{i,2}=\hat{\mathfrak{h}}_{v_i+n-2}, \quad W'_{i,1}=\hat{\mathfrak{h}}_{v_i+n-2}, \quad W'_{i,2}=\bar{\mathfrak{h}}_{v_i+1},\\
    Z_{i,1}=\bar{\mathfrak{h}}_{v_i+n-2},\quad Z_{i,2}=\hat{\mathfrak{h}}_{v_i+1}, \quad Z'_{i,1}=\hat{\mathfrak{h}}_{v_i+1}, \quad Z'_{i,2}=\bar{\mathfrak{h}}_{v_i+n-2},
\end{align*}
and $W_{i,j}=W'_{i,j}=Z_{i,j}=Z'_{i,j}=\mathfrak{h}_{v_i+j-2}$ for $j>2$. We claim the following
\begin{align}
    \label{eq: first}q \det(V)+\det(V')&=-\det(W)-q\det(W)'\\
    \label{eq: second}\det(W)&=-\det(Z)\\
    \label{eq: third}\det(W')&=-\det(Z')\\
    \label{eq: fourth}\det(Z)&=-q \det(Z')
\end{align}
which finishes the proof. 

To prove \eqref{eq: first}, pick any $1\leq a<b\leq n$. Then it is enough to show 
\begin{equation*}
    q \det\begin{pmatrix}
V_{a,1} & V_{a,2} \\
V_{b,1} & V_{b,2} 
\end{pmatrix}+\det\begin{pmatrix}
V'_{a,1} & V'_{a,2} \\
V'_{b,1} & V'_{b,2} 
\end{pmatrix}= - \det\begin{pmatrix}
W_{a,1} & W_{a,2} \\
W_{b,1} & W_{b,2} 
\end{pmatrix}-q \det\begin{pmatrix}
W'_{a,1} & W'_{a,2} \\
W'_{b,1} & W'_{b,2} 
\end{pmatrix}  
\end{equation*}
which is equivalent to (after a rearrangement of terms)
\begin{equation*}
 \left(   q \det\begin{pmatrix}
\hat{\mathfrak{h}}_{v_a} & \bar{\mathfrak{h}}_{v_a+1} \\
\hat{\mathfrak{h}}_{v_b+n-2} & \bar{\mathfrak{h}}_{v_b+n-1}  
\end{pmatrix}+\det\begin{pmatrix}
\bar{\mathfrak{h}}_{v_a+1}&\hat{\mathfrak{h}}_{v_a}   \\
\bar{\mathfrak{h}}_{v_b+n-1} &\hat{\mathfrak{h}}_{v_b+n-2}   
\end{pmatrix}\right) -\left(   q \det\begin{pmatrix}
\hat{\mathfrak{h}}_{v_b} & \bar{\mathfrak{h}}_{v_b+1} \\
\hat{\mathfrak{h}}_{v_a+n-2} & \bar{\mathfrak{h}}_{v_a+n-1}  
\end{pmatrix}+\det\begin{pmatrix}
\bar{\mathfrak{h}}_{v_b+1}&\hat{\mathfrak{h}}_{v_b}   \\
\bar{\mathfrak{h}}_{v_a+n-1} &\hat{\mathfrak{h}}_{v_a+n-2}   
\end{pmatrix}\right) =0
\end{equation*}
Each summand on the left-hand side is zero due the induction hypothesis for $n=2$.

Now we show \eqref{eq: second}. Starting from $W$, we apply following column operations:
\begin{itemize}
\item subtract the $n$-th column from the second column, then multiply (-1) to the second column
\item swap the first and the second columns
\item subtract the third column from the second column, then multiply (-1) to the second column.
\end{itemize}
It is easy to check that the resulting matrix is $Z$ and $\det(W)=-\det(Z)$ follows. Similar argument also shows \eqref{eq: third}.

Lastly, we show \eqref{eq: fourth}. Let $I=\{1,2,\cdots,n\}\setminus\{3,4\}$ and pick any $A\subset \{1,2,\cdots,n\}$ such that $|A|=n-2$. Then denoting $Z_{A,I}$ to be a submatrix given by restricting row indices to $A$ and column indices to $I$, induction hypothesis implies  $\det(Z_{A,I})+q \det(Z'_{A,I})=0$. This completes the proof. 
\end{proof}
\begin{corollary} With the same notations as in Lemma \ref{lem: q multi middle}, define $n$ by $n$ matrices $\hat{V}$ and $\hat{V'}$ by
\begin{itemize}
\item $\hat{V}_{i,j}=V_{i,j}$ and $\hat{V}'_{i,j}=V'_{i,j}$ for $j=1,2$
\item $\hat{V}_{i,j}=\hat{V}'_{i,j}=\hat{\mathfrak{h}}_{v_i+j-2}$ for $j>2$.
\end{itemize}
Then we have $q\det(\hat{V})+\det(\hat{V}')=0$.
\end{corollary}
\begin{proof}
Let $\Phi: \mathbb{F}[\textbf{y}]\rightarrow \mathbb{F}[\textbf{y}]$ be a linear operator defined by $\Phi(y^{A})=y^{A}$ if $\sum_{j}A_{0,j}=1$, otherwise $\Phi(y^{A})=0$. Then $\Phi$ sends $\left(q\det(V)+\det(V')\right)\cdot 1$ to $\left(q\det(\hat{V})+\det(\hat{V}')\right)\cdot 1$. Therefore we deduce $\left(q\det(\hat{V})+\det(\hat{V}')\right)\cdot 1=0$ and this is enough to imply $q\det(\hat{V})+\det(\hat{V}')=0$ as an operator. 
\end{proof}

\begin{proof}[Proof of Lemma \ref{lem: multiplication by q}] We define $n$ by $n$ matrices $V^{(1)}, V^{(2)}, \cdots, V^{(n)}$ given as
\begin{align*}
V^{(\ell)}_{i,j}=\begin{cases*}\hat{\mathfrak{h}}_{v_i+j-1}\qquad \text{if $j<\ell$}\\
 \bar{\mathfrak{h}}_{v_i+n-1} \qquad \text{if $j=\ell$}\\
\hat{\mathfrak{h}}_{v_i+j-2}\qquad \text{if $j>\ell$}.
\end{cases*}
\end{align*}
We claim that $\det(V^{(\ell)})=-q \det(V^{(\ell+1)})$. Let $I=\{\ell,\ell+1,\dots,n\}$ and pick any $A\subset \{1,2,\dots,n\}$ such that $|A|=n-\ell+1$. Then we have $q \det(V^{(\ell+1)}_{A,I})+ \det(V^{(\ell)}_{A,I})=0$ by the previous corollary. Therefore the claim is proved and we obtain $\det(V^{(1)})=(-q)^{n-1} \det(V^{(n)})$.

\end{proof}

\section{Generalized Macdonald polynomials for filled diagrams}
\label{sec5: column exchange filled diagram}
\subsection{Macdonald polynomials for filled diagrams and HHL formula}\label{subsec: generalization of Macdonald}

 Given a diagram $D$, we regard a word \(\sigma \in \mathbb{Z}_{\geq 1}^{{|D|}}\) as an assignment of integer entries to the cells of \( D \), proceeding from left to right within each row, and from the top row to the bottom row. In this way, we consider $\sigma(u)$ for $u\in D$ as an integer assigned for $u$ in the assignment of $\sigma$ in $D$. The \emph{content} $\cont_D(\sigma,i)$ of $\sigma$ in the $i$-th row of $D$ is defined by the (infinite) vector whose $a$-th component is given by the number of $a$'s in the assignment of $\sigma$ in the $i$-th row of $D$.

The celebrated result of Haglund--Haiman--Loehr \cite{HHL05} gives Macdonald polynomials using two statistics $\inv_\mu$ and $\maj_\mu$. We will use slightly different convention for $\inv$, which we now explain. For a diagram $D$, and cells $u=(i,j)$ and $v=(i',j')$ of $D$, we say that a pair $(u,v)$ is an \emph{attacking pair} if either 
\begin{center}
    (1) $i=i'$ and $j<j'$ $\quad$ or$\quad$ (2) $i=i'+1$ and $j>j'$. 
\end{center}
For a word $\sigma \in\mathbb{Z}_{\geq 1}^{|D|}$, we consider a pair $(u,v)$ of cells in $D$ to be an \emph{inversion pair} of $\sigma$ if $(u,v)$ is an attacking pair and $\sigma(u)>\sigma(v)$. We denote the set of inversion pairs of $\sigma$ by $\Inv_D(\sigma)$ and define $\inv_D(\sigma)$ to be 
\[
    \inv_D(\sigma) \coloneqq q^{|\Inv_D(\sigma)|}.
\]
A cell $u$ is \emph{descent} of $\sigma$ if $\sigma(u)>\sigma(v)$, where the cell $v$ is the cell right below $u$, i.e. $u=(i,j)$ and $v=(i-1,j)$. Define $\Des_D(\sigma)$ to be the set of descents of $\sigma$.

A \emph{filling} of $D$ is a function $f:D\rightarrow \mathbb{F}$, where we take $\mathbb{F}$ as a field containing $\mathbb{C}(q,t)$. Then a \emph{filled diagram} $(D,f)$ is a pair of a diagram and a filling on it. We define $\maj_{(D,f)}(\sigma)$ as the product of $f(u)$ over all cells $u$ which are descents of $\sigma$, i.e.,
\[
    \maj_{(D,f)}(\sigma) \coloneqq \prod_{u\in\Des_D(\sigma)} f(u).
\]
Note that bottom cells of $D$ cannot be a descent so the value $f(u)$ for bottom cell $u$ is redundant in a definition of $\maj_{(D,f)}$. Therefore we usually represent a filled diagram $(D,f)$ omitting values on bottom cells. Finally, $\stat_{(D,f)}(w)$ is defined by the product of $\inv_D$ and $\maj_{(D,f)}$,
\[
    \stat_{(D,f)}(\sigma) \coloneqq \inv_D(\sigma)\maj_{(D,f)}(\sigma).
\]

In \cite{KLO22}, a generalization of modified Macdonald polynomial of a filled diagram $(D,f)$ is defined by
\[
    \widetilde{H}_{(D,f)} := \sum_{\sigma \in \mathbb{Z}_{\geq 1}^{|D|}} \stat_{(D,f)}(\sigma) x^{\sigma}.
\]
We also call $\widetilde{H}_{(D,f)}$ as the \emph{Macdonald polynomial} for a filled diagram $(D,f)$. Indeed, this generalizes the concept of the celebrated HHL formula \cite{HHL05,HHL08}, which gives the combinatorial formula for the modified Macdonald polynomial. To state, we recall some notations first. Previously we defined $\arm_{\mu}(u)$ and  $\leg_{\mu}(u)$ for a partition $\mu$ and these were generalized to a more general shape \cite{HHL08}. 

Let $\beta$ be a diagram \(\beta = [[\beta_1], [\beta_2], \dots, [\beta_k]]\) for some positive integers $\beta_i$'s. For a cell \(u \in \beta\), the \(\arm_\beta\) and \(\leg_\beta\) are defined as follows:
\[
\leg_\beta(u) = \beta_i - j = (\text{number of cells strictly north of } u),
\]
\[
\arm_\beta(u) = |\{i' \in \{1, 2, \dots, i-1\} : j-1 \le \beta_{i'} < \beta_i\}| + |\{i' \in \{i, i+1, \dots, k\} : j \le \beta_{i'} \le \beta_i\}|.
\]
In other words, \(\arm_\beta(u)\) represents the number of cells strictly to the right of \(u\) in columns of height \(\beta_{i'} \leq \beta_i\) or strictly to the left of the cell just below \(u\) in columns of height \(\beta_{i'} < \beta_i\). Define the \emph{standard filling} $f^{\st}_{\beta}$ on $\beta$ by
\[
    f^{\st}_\beta(u) = q^{-\arm_\beta(u)}t^{\leg_\beta(u)}.
\]
Now we can give the celebrated HHL formula using this standard filling.

\begin{thm}\label{thm: HHL formula}\cite{HHL05, HHL08} 
Let $\mu$ be a partition and $\beta$ be the diagram obtained by rearranging columns of $\mu$. Then we have
\begin{equation}\label{eq: HHL formula}
    \widetilde{H}_\mu = \sum_{\sigma \in \mathbb{Z}_{\geq 1}^{|\beta|}} \stat_{(\beta,f_\beta^{\st})}(\sigma) x^{\sigma}.
\end{equation}
\end{thm}
Haglund, Haiman, and Loehr \cite{HHL05} proved \eqref{eq: HHL formula} for \(\beta = \mu\), and later generalized it to any rearrangements \(\beta\) in \cite{HHL08}. 

Theorem~\ref{thm: HHL formula} shows an example of two different filled diagrams giving the same Macdonald polynomials. Another (rather simple) operation on filled diagrams, called the \emph{cycling rule}, also gives the same Macdonald polynomials. Given a diagram \( D = [D^{(1)}, \dots, D^{(\ell)}] \), consider the diagram
\[
    D' = [D^{(2)}, \dots, D^{(\ell)}, D^{(1)} + 1],
\]
where \( I + 1 = \{a + 1 : a \in I\} \) for an interval \( I \). In other words, \( D' \) is obtained by moving the leftmost column of \( D \) to the end on the right and shifting it up by one cell. The cells in \( D \) and \( D' \) correspond naturally in a bijective manner. For a filling \( f \) of \( D \), we can define a filling \( f' \) of \( D' \) inherited from \( f \) through this natural bijection. We define \(\cycling(D, f) \coloneqq (D', f')\) and it is direct to see 
\begin{equation}\label{eq: cycling}
    \widetilde{H}_{(D, f)}[X] = \widetilde{H}_{(D', f')}[X].
\end{equation}

\subsection{Column exchange rule}

In \cite{KLO22}, the authors together with Seung Jin Lee introduced a \emph{column exchange rule} (Proposition \ref{prop: column exchange}), which we exploit in the next section. 

\begin{definition}\label{def: column exchange operator}
    Let $(D,f)$ be a filled diagram satisfying the following for some positive integer $j$:
    \begin{itemize}
    \item denoting $D=[D^{(1)},D^{(2)},\dots]$ there exist integers $n\geq m$ such that $D^{(j)}=[1,n]$ and $D^{(j+1)}=[1,m]$,
    \item if $n>m$ we have $\frac{f(2,j)}{f(2,j+1)}=\dots=\frac{f(m,j)}{f(m,j+1)}=q^{-1}{f(m+1,j)}$, and
    \item if $n=m$ we have  $\frac{f(2,j)}{f(2,j+1)}=\dots=\frac{f(m,j)}{f(m,j+1)}$.
    \end{itemize}
    Then we define $S_j(D,f)$ to be a filled diagram $(D',f')$ given by:
    $D'=[D^{(1)},\dots,D^{(j+1)},D^{(j)},\dots]$ and
    \begin{itemize}
    \item $f'(i,c)=f(i,c)$ if $c\neq j,j+1$,
    \item $f'(i,j)=f(i,j+1)$, and 
    \item $f'(i,j+1)=f(i,j)$ if $i\neq m+1$ and $f'(i,j+1)=q^{-1}f(i,j)$ if $i= m+1$.
    \end{itemize}
    As it is obvious that the map $S_j$ is injective, we can naturally define $S_j^{-1}$. We say that two filled diagrams are column-equivalent $(D,f)\equiv (D',f')$, if $(D',f')$ is obtained from $(D,f)$ by applying a sequence of maps of the form $S_j$ or $S_j^{-1}$. An example of the application of the map $S_j$ will be given in the next section (see bottom of Figure \ref{fig:rec}).
\end{definition}

\begin{proposition}\label{prop: column exchange}
    If $(D,f)\equiv (D',f')$ then there exists a bijection $\phi: \mathbb{Z}^{|D|}\rightarrow \mathbb{Z}^{|D'|}$ satisfying
    \begin{itemize}
    \item $(\phi1)$ $\stat_{(D,f)}(\sigma)=\stat_{(D',f')}(\phi(\sigma))$  \item $(\phi2)$ $\cont_{D}(\sigma,i)=\cont_{D'}(\phi(\sigma),i)$ for all $i$.
    \item $(\phi3)$ $\iDes(\std(\sigma))=\iDes(\std(\phi(\sigma))$ where $\std$ denotes the standardization of a word.
    \end{itemize}
\end{proposition}

Indeed, if \((D,f) \equiv (D',f')\), property \((\phi1)\) together with the condition \(x^{\sigma} = x^{\phi(\sigma)}\), a weaker condition than \((\phi2)\), are sufficient to imply \(\widetilde{H}_{(D,f)} = \widetilde{H}_{(D',f')}\). The map \(\phi\) also satisfies \((\phi2)\) and \((\phi3)\). Later, we use these additional properties to derive a more refined identity, see \eqref{eq: column exchange refined example} for example.

\section{Proof of Theorem \ref{thm: main theorem} (c)}
\label{sec: deform}
In this section, we show Theorem \ref{thm: main theorem} (c). In particular we will connect $\left(e_{|\mu|-|\lambda|-k}^\perp \I_{\mu,\tilde{\lambda},k}\right)$ with the Loehr--Warrington formula through \eqref{eq: lw formula final}.  By the shape independence in Theorem~\ref{thm: main theorem} (b), it is sufficient to provide a combinatorial formula for $\left(e_{|\mu|-|\lambda|-k}^\perp \I_{\mu,\tilde{\lambda},k}\right)$ for a specific partition \(\mu\) of \(n\) corners.  We fix \(\mu\) to be the augmented staircase $\delta=\delta_{n,N} \coloneqq (\underbrace{n,\dots,n}_{N},n-1,\dots,1)$ for a large enough $N$ ($N>|\lambda|$ suffices). We outline the proof strategy: 
\begin{itemize}
    \item In Section \ref{sub: 6.1}, we define auxiliary filled diagrams $(\delta_{S},f_{S})$, $(\delta_{S},f^{z}_{S})$ and $(\Rec^{\leq m},g^{z}_S)$ 
    \item In Section \ref{sub 6.2}, we go over a technical process to obtain \eqref{eq: rec final} from $\left(e_{|\delta|-|\lambda|-k}^\perp \I_{\delta,\tilde{\lambda},k}\right)$, imitating arguments in \cite[Section 6]{KLO23}
    \item In Section \ref{sub: 6.3}. we illustrate a clear connection between \eqref{eq: rec final} and (reformulated) Jacobi-Trudi type formula \eqref{eq: lw formula final} for the Loehr--Warrington formula $\LW_\lambda$.
\end{itemize}

\subsection{Deformation of filled diagrams}\label{sub: 6.1}

Let the corners of $\delta$ be indexed by \(c_1, c_2, \dots, c_n\) from top to bottom. Additionally, for a $k$-subset $S=\{i_1<i_2<\cdots<i_k\}$ of \([n]\), let \(\delta^{S}\) denote the partition obtained from $\delta$ by deleting the corners $c_i$'s for $i
\in S$. We rearrange columns of $\delta^{S}$ by moving the $i_1, i_{2}, \dots, i_k$-th columns all the way to the left in this order and denote by $\beta$. Now we define $(\delta_S,f_S):=\cycling^{k}(\beta,f^{\st}_{\beta})$ (the diagram $\delta_S$ should not be confused with $\delta^{S}$). According to Theorem~\ref{thm: HHL formula} and \eqref{eq: cycling}, we may use $(\delta_S,f_S)$ for the computation of $\widetilde{H}_{\mu^{S}}$ i.e.  $\widetilde{H}_{\mu^{S}}= \widetilde{H}_{(\delta_S,f_S)}$. For example, Figure \ref{fig: delta S process} illustrates the described process for $\delta=\delta_{4,3}=(4,4,4,3,2,1)$ and $S=\{2,3\}$. The first figure shows $(\delta^{S},f^{\st}_{\delta^{S}})$ and the second shows $(\beta,f^{\st}_{\beta})$ for the described rearrangement $\beta$. Applying the operator $\cycling$ twice, we obtain $(\delta_{S},f_{S})$ as shown in the third figure. 

Now, we introduce indeterminates $z_1,z_2,\dots,$ and define $(\delta_S,f^{z}_S)$ which is a $z$-deformation $(\delta_S,f_S)$. Denote $S=\{i_1<i_2<\dots<i_k\}$ and $S^{c}=\{j_1<j_2<\dots<j_{n-k}\}$. Then the filling $f^{z}_{S}$ is given by 
\begin{equation*}
    f^{z}_S(a,b)=q\frac{z_{N+n+1-a}}{z_{j_b}} \quad \text{for  $b\leq n-k$,} \quad f^{z}_S(a,n-k+b)=\begin{cases*}
        \frac{z_{N+n+1-a}}{z_{i_b}}\qquad \text{if $N+n+1-a\in S^{c}$ }\\
        q\frac{z_{N+n+1-a}}{z_{i_b}}\qquad \text{otherwise}
    \end{cases*} \quad \text{for $b\leq k$}.
\end{equation*}

The fourth figure in Figure~\ref{fig: delta S process} shows $(\delta_S,f^{z}_S)$ corresponding to $(\delta_S,f_S)$ in the third.
A straight calculation shows that under the specialization given by
\begin{equation}\label{eq: z_i values}
    z_j = 
    \begin{cases}
        q^{1-j}t^{j+1-N-n}&\text{if } j\le n\\
        q^{-n}t^{j+1-N-n}&\text{if } j > n
    \end{cases}
\end{equation}
$(\delta_S,f^{z}_S)$ recovers as  $(\delta_S,f_S)$.

\ytableausetup
{boxsize=2.5em}
 \begin{figure}[ht]
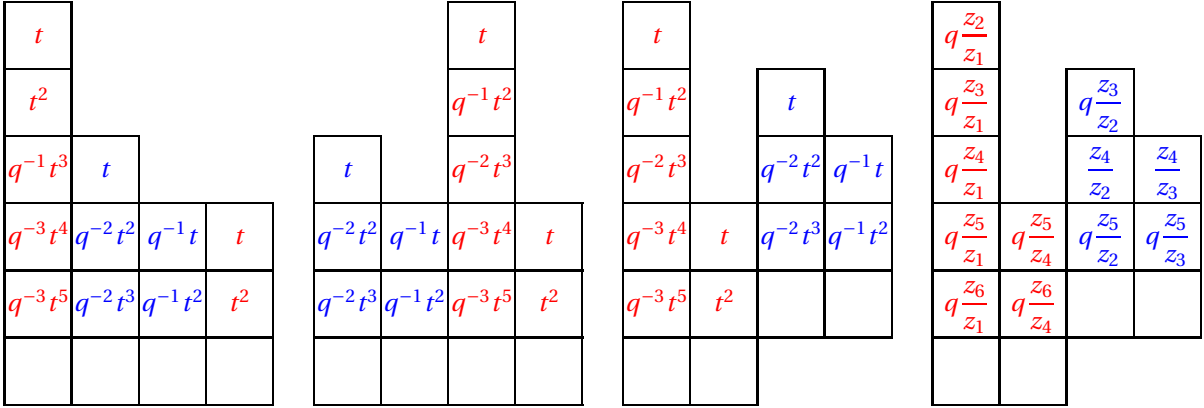
 
\scalebox{1.0}{\begin{ytableau}
 \textcolor{red}{t} \\
 \textcolor{red}{t^2} \\
 \textcolor{red}{q^{-1}t^3} & \textcolor{blue}{t}\\
 \textcolor{red}{q^{-3}t^4} & \textcolor{blue}{q^{-2}t^2} & \textcolor{blue}{q^{-1}t} &  \textcolor{red}{t}\\
 \textcolor{red}{q^{-3}t^5} & \textcolor{blue}{q^{-2}t^3} & \textcolor{blue}{q^{-1}t^2} & \textcolor{red}{t^2}\\&&&\\
\end{ytableau}}
\quad
\scalebox{1.0}{ \begin{ytableau}
 \none & \none & \textcolor{red}{t} \\
 \none &\none & \textcolor{red}{q^{-1}t^2} \\
 \textcolor{blue}{t} & \none & \textcolor{red}{q^{-2}t^3} \\
 \textcolor{blue}{q^{-2}t^2} &  \textcolor{blue}{q^{-1}t} & \textcolor{red}{q^{-3}t^4} &  \textcolor{red}{t} \\
 \textcolor{blue}{q^{-2}t^3} & \textcolor{blue}{q^{-1}t^2} & \textcolor{red}{q^{-3}t^5} & \textcolor{red}{t^2} \\ &&& \\
 \end{ytableau}}
 \quad
 \scalebox{1.0}{
 \begin{ytableau}
 \textcolor{red}{t} \\
 \textcolor{red}{q^{-1}t^2} & \none & \textcolor{blue}{t} \\
 \textcolor{red}{q^{-2}t^3} & \none & \textcolor{blue}{q^{-2}t^2} &  \textcolor{blue}{q^{-1}t}\\
 \textcolor{red}{q^{-3}t^4} &  \textcolor{red}{t} & 
 \textcolor{blue}{q^{-2}t^3} & \textcolor{blue}{q^{-1}t^2} \\
 \textcolor{red}{q^{-3}t^5} & 
 \textcolor{red}{t^2} & & \\ &
 \end{ytableau}
 }
 \quad
\scalebox{1.0}{\begin{ytableau}
 \textcolor{red}{q\dfrac{z_2}{z_1}} \\
 \textcolor{red}{q\dfrac{z_3}{z_1}} & \none & \textcolor{blue}{q\dfrac{z_3}{z_2}} \\
 \textcolor{red}{q\dfrac{z_4}{z_1}} & \none & \textcolor{blue}{\dfrac{z_4}{z_2}} &  \textcolor{blue}{\dfrac{z_4}{z_3}}\\
 \textcolor{red}{q\dfrac{z_5}{z_1}} &  \textcolor{red}{q\dfrac{z_5}{z_4}} & \textcolor{blue}{q\dfrac{z_5}{z_2}} & \textcolor{blue}{q\dfrac{z_5}{z_3}} \\
 \textcolor{red}{q\dfrac{z_6}{z_1}} & \textcolor{red}{q\dfrac{z_6}{z_4}} & & \\  &  \\
 \end{ytableau}}
 \caption{The process to obtain $(\delta_{S},f_{S}^{z})$.}
 \label{fig: delta S process}
 \end{figure}

Lastly for a positive integer $m\leq N$, we define $(\Rec^{\leq m},g_{S}^{z})$ to be a sub filled diagram of $(\delta_S,f^{z}_S)$ obtained by restricting to the first $m$ rows. Note that the diagram $\Rec^{\leq m}=[\underbrace{[1,m],\dots,[1,m]}_{n-k},\underbrace{[2,m],\dots,[2,m]}_{k}]$ is independent of the choice of $S$. In Figure \ref{fig:rec}, top figure shows $(\Rec^{\leq m},g_{S}^{z})$ obtained from the fourth figure in Figure \ref{fig: delta S process}. 

For $S\in \binom{n}{k}$, the filled diagrams $(\delta_S,f^{z}_S)$ are not in general column-equivalent to each other. However, after truncating bottom rows, they are column-equivalent as shown in the following lemma. 

\ytableausetup
{boxsize=2.5em}
\begin{figure}[ht]
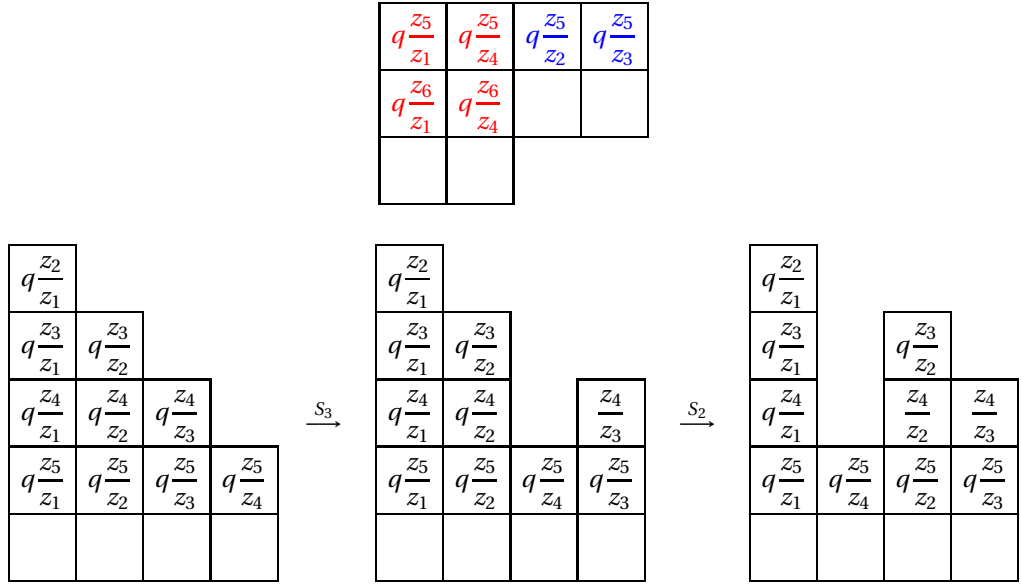

    \centering
    \begin{ytableau}
     \textcolor{red}{q\dfrac{z_5}{z_1}} &  \textcolor{red}{q\dfrac{z_5}{z_4}} & \textcolor{blue}{q\dfrac{z_5}{z_2}} & \textcolor{blue}{q\dfrac{z_5}{z_3}} \\
     \textcolor{red}{q\dfrac{z_6}{z_1}} & \textcolor{red}{q\dfrac{z_6}{z_4}} & & \\  &  \\
     \end{ytableau}
     
     \vspace{0.5cm}
     \begin{ytableau}
     q\dfrac{z_2}{z_1} \\
     q\dfrac{z_3}{z_1} & q\dfrac{z_3}{z_2} \\
     q\dfrac{z_4}{z_1} & q\dfrac{z_4}{z_2} & q\dfrac{z_4}{z_3} \\ q\dfrac{z_5}{z_1} & q\dfrac{z_5}{z_2} & q\dfrac{z_5}{z_3} & q\dfrac{z_5}{z_4} \\
     &&&
     \end{ytableau}\quad\raisebox{-45pt}{$\overset{S_3}{\longrightarrow}$}
     \quad
     \begin{ytableau}
     q\dfrac{z_2}{z_1} \\
     q\dfrac{z_3}{z_1} & q\dfrac{z_3}{z_2} \\
     q\dfrac{z_4}{z_1} & q\dfrac{z_4}{z_2} & \none & \dfrac{z_4}{z_3} \\ q\dfrac{z_5}{z_1} & q\dfrac{z_5}{z_2} & q\dfrac{z_5}{z_4} & q\dfrac{z_5}{z_3} \\
     &&&
     \end{ytableau} \quad\raisebox{-45pt}{$\overset{S_2}{\longrightarrow}$}
     \quad 
     \begin{ytableau}
     q\dfrac{z_2}{z_1} \\
     q\dfrac{z_3}{z_1} & 
     \none & 
     q\dfrac{z_3}{z_2} \\
     q\dfrac{z_4}{z_1} & 
     \none & 
     \dfrac{z_4}{z_2} & \dfrac{z_4}{z_3} \\ q\dfrac{z_5}{z_1} & q\dfrac{z_5}{z_4} & q\dfrac{z_5}{z_2} & q\dfrac{z_5}{z_3} \\
     &&&
     \end{ytableau} 

        \caption{An example of $\Rec^{\leq m}$ (top figure) and an illustration of the proof of Lemma \ref{lem: necessary lemma for filled diagrams} (bottom figures).}
    \label{fig:rec}
\end{figure}

\begin{lem}\label{lem: necessary lemma for filled diagrams}
For $m\geq 2$, let $(D'_{S},f'_{S})$ be a sub filled diagram of $(\delta_{S},f_{S})$ obtained by restricting to the $m$-th row and above. Then for any $I, J\in \binom{[n]}{k}$ we have
\begin{equation*}
    (D'_{I},f'_{I})\equiv(D'_{J},f'_{J}).
\end{equation*}
\end{lem}
\begin{proof}It suffices to prove the case when $m=2$. Consider a filled diagram $(D,f)$ given by $D = (\underbrace{n, \dots, n}_{N-1}, n-1, \dots, 1)$ and $f(a,b) = \frac{q z_{N+n-a}}{z_j}$. Then $(D'_S, f'_S) \equiv (D, f)$ for any $S \in \binom{[n]}{k}$, as illustrated in the bottom of Figure \ref{fig:rec}. The left side shows $(D,f)$, and after applying the operators $S_3$ and $S_2$, we obtain the filled diagram on the right. This diagram is derived from $(\delta_S, f_S^z)$ in Figure \ref{fig: delta S process} by deleting the first row.
\end{proof}

\subsection{Reducing to the $\Rec^{\leq N}$}\label{sub 6.2}
In this section we compute $\left(e_{|\delta|-k-|\lambda|}^\perp \I_{\delta,\tilde{\lambda},k}\right)$ starting from  filled diagrams $(\delta_S,f_{S})$ for the computation of $\widetilde{H}_{\delta^{S}}$. Our goal is to obtain \eqref{eq: rec final}. This section consists of a technical process where the idea of most of the arguments was given in \cite[Section 6]{KLO23}.
\begin{definition}
    Given a diagram $D$ and a positive integer $\ell \leq |D|$, we define $\Sub(D,\ell)$ to be the set of assignments of $\ell$ cells of $D$ by positive integers. We think of $R\in \Sub(D,\ell)$ as a function $R: D'\rightarrow \mathbb{Z}_{\geq 1}$ for some subdiagram $D'\subseteq D$ with $|D'|=\ell$. We let $\rsum(R)$ (resp $\csum(R)$) to be a vector whose $i$-th entry is a number of cells in the $i$-th row (resp column) of $D'$. We also define an induced full assignment $\bar{R}:D \rightarrow \mathbb{Z}_{\geq 1}$ given by assigning $1,2,\dots,|D|-\ell$ for cells in $D\setminus D'$ in a reverse reading order, and by letting $\bar{R}(c)=R(c)+|D|-\ell$ for cells $c\in D'$. See Figure \ref{fig: sub filling} for an example. We let $x^{R}=\prod_{c\in D'}x_{R(c)}$ and lastly for a filled diagram $(D,f)$, we abuse the notation for $\stat_{(D,f)}(R)$ instead of $\stat_{(D,f)}(\bar{R})$.
\end{definition}

\begin{figure}[ht]
    \centering
    \ytableausetup{boxsize=2.5em, boxframe=0.1em}
    \begin{ytableau}
    \textcolor{white}{a} &&& 1 \\
    4 & && 1 \\
    3 & 
    \end{ytableau}
    \qquad
    \begin{ytableau}
    \textcolor{blue}{6} & \textcolor{blue}{5} & \textcolor{blue}{4} & 7 \\
    10 & \textcolor{blue}{3} & \textcolor{blue}{2} & 7 \\
    9 & \textcolor{blue}{1}
    \end{ytableau}
    \caption{The left shows $R\in \Sub(D,4)$ for $D=\Rec^{\leq 3}$ and the right show the induced full assignment $\bar{R}$. We have $\rsum(R)=(1,2,1)$, $\csum=(2,0,0,2)$ and $x^{R}=x_1^2x_3x_4$.}
    \label{fig: sub filling}
\end{figure}

\begin{lem}
We have    \begin{equation}\label{eq: after perp}
        e_{|\delta|-k-|\lambda|}^\perp \I_{\delta,\tilde{\lambda},k} =(-1)^{|\tilde{\lambda}|}\sum_{S \in \binom{[n]}{k}} \dfrac{s_{\tilde{\lambda}}[w_S] \prod_{j \in S^c}{w_j}}{\prod_{\substack{i \in S \\ j \in S^c}} (w_j - w_i)} \sum_{R\in \Sub(\delta_{S},|\lambda|)}\stat_{(\delta_{S},f_{S})}(R)x^{R}
    \end{equation}
    where $w_i=T_{\delta^{(i)}}/T_{\delta}$.
\end{lem}
\begin{proof}
For any nonnegative integer $m$, it is elementary to see
\begin{equation*}
    e^{\perp}_{|\delta_{S}|-k-m} \left(\sum_{\sigma \in \mathbb{Z}^{|\delta_{S}|}_{\geq 1}}\stat_{(\delta_{S},f_{S})}(\sigma)x^{\sigma}\right)=\sum_{R\in \Sub(\delta_{S},m)}\stat_{(\delta_{S},f_{S})}(R)x^{R},
\end{equation*}
and the proof follows.
\end{proof}
For indeterminates $z_1,z_2,\dots$ we consider the $z$-deformation of the right hand side of \eqref{eq: after perp}:
\begin{equation}\label{eq: z deform initial}
    (-1)^{|\tilde{\lambda}|}\sum_{S \in \binom{[n]}{k}} \dfrac{s_{\tilde{\lambda}}[z_S] \prod_{j \in S^c}{z_j}}{\prod_{\substack{i \in S \\ j \in S^c}} (z_j - z_i)} \sum_{R\in \Sub(\delta_{S},|\lambda|)}\stat_{(\delta_{S},f^{z}_{S})}(R)x^{R}.
\end{equation}
Since $w_i=T_{\delta^{(i)}}/T_{\delta}=q^{1-i}t^{i+1-N-n}$, we have $w_i=z_i$ under the specialization \eqref{eq: z_i values}. Therefore we have
\begin{equation}\label{eq: up toconstants}
    \eqref{eq: z deform initial} \xrightarrow[\text{Specialize by \eqref{eq: z_i values}}]{\text{}}   \left(e_{|\delta|-k-|\tilde{\lambda}|}^\perp \I_{\delta,\tilde{\lambda},k}\right)  
\end{equation}
Now our goal is to reduce \eqref{eq: z deform initial} to \eqref{eq: rec final}.

The following lemma follows from a simple computation so we leave it as an exercise to readers.
\begin{lem}\label{lem: for the longest permutation}
Let $w'_0$ and $w_0$ be the longest permutations of length $|\delta_S|$ and $|\Rec^{\leq N}|$. The following are true:

(1) the quantity $(\prod_{j \in S^c}{z_j})\stat_{(\delta_{S},f^{z}_{S})}(w'_0)$ is independent of the choice of $S\in \binom{[n]}{k}$

(2) the quantity $(\prod_{j \in S^c}{z_j})\stat_{(\Rec^{\leq N},g^{z}_{S})}(w_0)$ is independent of the choice of $S\in \binom{[n]}{k}$.

(3) the quantity $(\prod_{j \in S^c}{z_j})\stat_{(\delta_{S},f^{z}_{S})}(w'_0)$ specializes to $T_{\delta^{[n]}}$ under \eqref{eq: z_i values}.
\end{lem}
For the rest of the paper, we denote the quantity in Lemma \ref{lem: for the longest permutation} (1) and (2) by $C_1$ and $C_2$ respectively. 
\begin{definition}
Let $D=\Rec^{\leq m}$ for some $m\leq N$ and consider $R\in \Sub(D,\ell)$ denoted by $R:D'\rightarrow \mathbb{Z}_{>0}$. We say that a cell $c=(i,j)\in D'$ is bad if $c'=(i+1,j)\in \Des_{D}(\bar{R})$ where $\bar{R}$ is the associated full assignments. Note that every cell $c\in D'$ located in a top row is bad. We define $\gamma(R)$ to be a vector whose $i$-th entry is the number of bad cells in the $i$-th row. For example, we have $\gamma(R)=(1,0,1)$ for $R$ in Figure \ref{fig: sub filling}
\end{definition}

Now we prove Lemma \ref{lem: rec vanishing} and \ref{lem: consider rec}, key ingredients to obtain \eqref{eq: rec final}.
\begin{lem}\label{lem: rec vanishing}For vectors $\alpha$ and $\gamma$ in $\mathbb{Z}_{\geq 0}^{m}$, we have
\begin{enumerate}
\item $\left(\sum_{\substack{R\in \Sub(\Rec^{\leq m},\ell)\\ \rsum(R)=\alpha, \gamma(R)=\gamma}}\stat_{(\Rec^{\leq m},g^{z}_{S})}(R)x^{R}\right)$ is symmetric in variables $z_S$ and symmetric in variables $z_{S^{c}}$,
\item if $|\alpha|-|\gamma|<|\lambda|$ we have 
\begin{equation}\label{eq: rec vanishing}\sum_{S \in \binom{[n]}{k}} \dfrac{s_{\tilde{\lambda}}[z_S] \prod_{j \in S^c}{z_j}}{\prod_{\substack{i \in S \\ j \in S^c}} (z_j - z_i)} \sum_{\substack{R\in \Sub(\Rec^{\leq m},\ell)\\ \rsum(R)=\alpha, \gamma(R)=\gamma}}\stat_{(\Rec^{\leq m},g^{z}_{S})}(R)x^{R}=0.\end{equation}
\end{enumerate}
\end{lem}
\begin{proof}We denote $A_{\alpha,S}$ and $A_{\alpha,S}^{\gamma}$ to be
\begin{equation*}
    A_{\alpha,S}=\sum_{\substack{R\in \Sub(\Rec^{\leq m},\ell)\\ \rsum(R)=\alpha}}\stat_{(\Rec^{\leq m},g^{z}_{S})}(R)x^{R}, \qquad  A_{\alpha,S}^{\gamma}=\sum_{\substack{R\in \Sub(\Rec^{\leq m},\ell)\\ \rsum(R)=\alpha, \gamma(R)=\gamma}}\stat_{(\Rec^{\leq m},g^{z}_{S})}(R)x^{R}.
\end{equation*} 
For $S=\{i_1<\dots<i_k\}$ and $S^{c}=\{j_1<\dots<j_{n-k}\}$, we have $A^{\gamma}_{\alpha,S}=A^{\gamma}_{\alpha,[n-k+1,n]}(z_{j_1},\dots,z_{j_{n-k}},z_{i_1},\dots,z_{i_k})$ by the property of $g_{S}^{z}$. Therefore it is enough to show (1) when $S=[n-k+1,n]$. Set $S=[n-k+1,n]$ and we first show that $A_{\alpha,S}$ is symmetric in variables $z_S$ and symmetric in variables $z_{S^{c}}$. Let $(D',g')$ be the filled diagram obtained by applying the map $S_1$ (see Definition \ref{def: column exchange operator}) to $(\Rec^{\leq m},g^{z}_{S})$. Then $D'=\Rec^{\leq m}$ and the filling $g'$ can be obtained from $g^{z}_{S}$ by swapping $z_1$ and $z_2$. By Proposition \ref{prop: column exchange}, we have
\begin{equation}\label{eq: column exchange refined example}
    \sum_{\substack{R\in \Sub(\Rec^{\leq m},\ell)\\ \rsum(R)=\alpha}}\stat_{(\Rec^{\leq m},g^{z}_{S})}(R)x^{R}=\sum_{\substack{R\in \Sub(D',\ell)\\ \rsum(R)=\alpha}}\stat_{(D',g')}(R)x^{R}.
\end{equation}
The right hand side is obtained from the left hand side by swapping $z_1$ and $z_2$, this shows $A_{\alpha,S}$ is symmetric in $z_1$ and $z_2$. Applying the similar argument, we conclude that $A_{\alpha,S}$ is symmetric in variables $z_S$ and symmetric in variables $z_{S^{c}}$. 

Now pick any $R\in \Sub(\Rec^{\leq m},\ell)$ with $\rsum(R)=\alpha$ and $\gamma(R)=\gamma$. On the $(i+1)$-th row there are $(\alpha_i-\gamma_i)$ many cells that do not belong to $\Des_{\Rec^{\leq m}}(\bar{R})$. Since the filling $g^{z}_{S}(i+1,j)$ on the $(i+1)$-th row is of the form $\frac{q z_{N+n-i}}{z_a}$ for some $1\leq a\leq n$, we can write
\begin{equation*}
    \stat_{(\Rec^{\leq m},g^{z}_{S})}(R)=\stat_{(\Rec^{\leq m},g^{z}_{S})}(w_0)\left(\frac{q^{c}\prod_{i=1}^{n}z_i^{d_i}}{\prod_{i=1}^{N-1}z_{N+n-i}^{\alpha_i-\gamma_i}}\right)
\end{equation*}
for some $c$ and $d_i$'s, where $w_0$ is the longest permutation of length $|\Rec^{\leq m}|$. We conclude that there exsits a polynomial $P_{\alpha,S}^{\gamma}(z_1,z_2,\dots,z_n)$ of degree $(|\alpha|-|\gamma|)$ such that
\begin{equation*}
    A_{\alpha,S}^{\gamma}=\stat_{(\Rec^{\leq m},g^{z}_{S})}(w_0)\left(\frac{P_{\alpha,S}^{\gamma}}{\prod_{i=1}^{N-1}z_{N+n-i}^{\alpha_i-\gamma_i}}\right).
\end{equation*}
 From 
\begin{equation*}
    \frac{A_{\alpha,S}}{\stat_{(\Rec^{\leq m},g^{z}_{S})}(w_0)}= \frac{\sum_{\gamma}A_{\alpha,S}^{\gamma}}{\stat_{(\Rec^{\leq m},g^{z}_{S})}(w_0)}=\sum_{\gamma}\frac{P_{\alpha,S}^{\gamma}}{\prod_{i=1}^{N-1}z_{N+n-i}^{\alpha_i-\gamma_i}},
\end{equation*}
the left hand side is symmetric in variables $z_S$ and symmetric in variables $z_{S^{c}}$, which implies that each $P_{\alpha,S}^{\gamma}$  is symmetric in variables $z_S$ and symmetric in variables $z_{S^{c}}$. This establishes (1).

To prove (2), from the symmetry established in (1), we may simply write $P_{\alpha,S}^{\gamma}=P_{\alpha}^{\gamma}(z_{S^{c}},z_{S})$ for $P_{\alpha}^{\gamma}=P_{\alpha,[n-k+1,n]}^{\gamma}$. Since $C=\prod_{j\in S^{c}}z_j\left(\stat_{(\Rec^{\leq m},g^{z}_{S})}(w_0)\right)$ is independent of the choice of $S\in\binom{[n]}{k}$ by Lemma \ref{lem: for the longest permutation} (2), the left hand side of \eqref{eq: rec vanishing} equals 
\begin{equation*}
    \frac{C}{\prod_{i=1}^{N-1}z_{N+n-i}^{\alpha_i-\gamma_i}}\left(\sum_{S \in \binom{[n]}{k}} \dfrac{s_{\tilde{\lambda}}[z_S] P_{\alpha}^{\gamma}(z_S^{c},z_{S})}{\prod_{\substack{i \in S \\ j \in S^c}} (z_j - z_i)}\right).
\end{equation*}
Note that $P_{\alpha}^{\gamma}$ is of degree $|\alpha|-|\gamma|$ (in variables $z_1,z_2,\dots,z_n$). Therefore if $|\alpha|-|\gamma|<|\lambda|$, the above vanishes by Lemma \ref{lem: Schur orthogonal in rectangle}.
\end{proof}

\begin{lem}\label{lem: consider rec}
Let $\alpha$ be a vector such that $\alpha_m=0$ for some $m\leq N$ and $\sum_{i=1}^{m-1}\alpha_i<|\lambda|$. Then we have 
\begin{equation}\sum_{S \in \binom{[n]}{k}} \dfrac{s_{\tilde{\lambda}}[z_S] \prod_{j \in S^c}{z_j}}{\prod_{\substack{i \in S \\ j \in S^c}} (z_j - z_i)} \sum_{\substack{R\in \Sub(\delta_{S},\ell)\\\rsum(R)=\alpha}}\stat_{(\delta_{S},f^{z}_{S})}(R)x^{R}=0.
    \end{equation}
\end{lem}
\begin{proof}
Recall that $(\Rec^{\leq m},g^{z}_{S})$ is a sub filled diagram of $(\delta_{S},f^{z}_{S})$ given by restricting to the first $m$ rows. We denote $(D_S,f'_S)$ to be a remaining sub filled diagram of $(\delta_{S},f^{z}_{S})$ (obtained by restricting to the $(m+1)$-th row and above). Note that $R\in \Sub(\delta_{S},\ell)$ naturally induces $R_1\in \Sub(\Rec^{\leq m},\ell')$ and $R_2\in \Sub(D_{S},\ell'')$ by restriction. Since $\alpha_m=0$, there exists a constant $C$ (independent of $S$) such that
\begin{equation*}
    \stat_{(\delta_{S},f^{z}_{S})}(R)=C\stat_{(\Rec^{\leq m},g^{z}_{S})}(R_1)\stat_{(D_S,f'_S)}(R_2).
\end{equation*}
Therefore we have 
\begin{equation*}
  \sum_{\substack{R\in \Sub(\delta_{S},\ell)\\\rsum(R)=\alpha}}\stat_{(\delta_{S},f^{z}_{S})}(R)x^{R}=C\left(\sum_{\substack{R_1\in \Sub(\Rec_{\leq m},\ell')\\\rsum(R_1)=\alpha^{(1)}}}\stat_{(\Rec^{\leq m},g^{z}_{S})}(R_1)x^{R_1}\right)\left(\sum_{\substack{R_2\in \Sub(D_S,\ell'')\\\rsum(R_2)=\alpha^{(2)}}}\stat_{(D_{S},f'_{S})}(R_2)x^{R_2}\right)  
\end{equation*}
where $\alpha^{(1)}=(\alpha_1,\alpha_2,\dots,\alpha_m)$ and $\alpha^{(2)}=(\alpha_{m+1},\dots)$. Note that $\left(\sum_{\substack{R_2\in \Sub(D_S,\ell'')\\\rsum(R_2)=\alpha^{(2)}}}\stat_{(D_{S},f'_{S})}(R_2)x^{R_2}\right)$ is independent of the choice of $S$ by Lemma \ref{lem: necessary lemma for filled diagrams} and Proposition \ref{prop: column exchange}. Now Lemma \ref{lem: rec vanishing} (2) completes the proof as $|\alpha^{(1)}|<|\lambda|$.
\end{proof}
By the previous lemma, in \eqref{eq: z deform initial}, it is enough to consider $R\in \Sub(\delta_{S},|\lambda|)$ such that $\rsum(R)=\alpha$ satisfies $\alpha_i=0$ for $i\geq N$. Otherwise we have $\sum_{i=1}^{N-1}{\alpha_i}<|\lambda|$ and there exists $\alpha_m=0$ for $m\leq N$ as $N>|\lambda|$. We can naturally regard such $R$ as an element of $ \Sub(\Rec^{\leq N},|\lambda|)$ by restriction. The transition ratio of $\stat$ is given by
\begin{equation*}
    \frac{C_1}{C_2}=\frac{\stat_{(\delta_S,f^{z}_{S})}(w'_0)}{\stat_{(\Rec^{\leq N},g^{z}_{S})}(w_0)}=\frac{\stat_{(\delta_S,f^{z}_{S})}(R)}{\stat_{(\Rec^{\leq N},g^{z}_{S})}(R)}.
\end{equation*}
 Therefore \eqref{eq: z deform initial} becomes
\begin{equation*}
    (-1)^{|\tilde{\lambda}|} \frac{C_1}{C_2}  \sum_{S \in \binom{[n]}{k}} \dfrac{s_{\tilde{\lambda}}[z_S] \prod_{j \in S^c}{z_j}}{\prod_{\substack{i \in S \\ j \in S^c}} (z_j - z_i)} \sum_{\substack{R\in \Sub(\Rec^{\leq N},|\lambda|)}}\stat_{(\Rec^{\leq N},g^{z}_{S})}(R)x^{R}
\end{equation*}
and by Lemma \ref{lem: rec vanishing} (2), this equals to 
\begin{equation}\label{eq: rec final}
    (-1)^{|\tilde{\lambda}|}\frac{C_1}{C_2}  \sum_{S \in \binom{[n]}{k}} \dfrac{s_{\tilde{\lambda}}[z_S] \prod_{j \in S^c}{z_j}}{\prod_{\substack{i \in S \\ j \in S^c}} (z_j - z_i)} \sum_{\substack{R\in \Sub(\Rec^{\leq N},|\lambda|)\\ \gamma(R)=0}}\stat_{(\Rec^{\leq N},g^{z}_{S})}(R)x^{R}.
\end{equation}

\subsection{Toward the Jacobi-Trudi type formula}\label{sub: 6.3}
In this section, we complete the proof of Theorem \ref{thm: main theorem} (c), starting from \eqref{eq: rec final}. Let $B_S$ to be 
\begin{equation*}
    B_S=\frac{1}{\stat_{(\Rec^{\leq N},g^{z}_{S})}(w_0)}\sum_{\substack{R\in \Sub(\Rec^{\leq N},|\lambda|)\\ \gamma(R)=0}}\stat_{(\Rec^{\leq N},g^{z}_{S})}(R)x^{R}.
\end{equation*}
Then \eqref{eq: rec final} equals to 
\begin{equation}\label{eq: rec final2}
    (-1)^{|\tilde{\lambda}|}C_1 \times \sum_{S \in \binom{[n]}{k}} \dfrac{s_{\tilde{\lambda}}[z_S] B_S}{\prod_{\substack{i \in S \\ j \in S^c}} (z_j - z_i)}. 
\end{equation}
Note that $B_S$ is symmetric in variables $z_S$ and symmetric in variable $z_{S^{c}}$ by Lemma \ref{lem: for the longest permutation} (2) and Lemma \ref{lem: rec vanishing} (1). Therefore letting $B=B_{[n-k+1,n]}$, we can  write $B_S=B(z_{S^{C}},z_S)$. The following lemma says that understanding a monomial coefficient of $B$ is crucial.
\begin{lem}\label{lem: key connection JT formula} Let $\lambda$ be a partition inside $(n-k) \times k$ rectangle and $P(z_1,z_2,\dots,z_n)$ be a polynomial function of degree $|\lambda|$ that is symmetric in variables $z_1,z_2,\dots,z_{n-k}$ and symmetric in variables $z_{n-k+1},z_{n-k+2},\dots,z_{n}$. Consider an $n$ by $n$ matrix $V$ given by 
\begin{align*}
    V_{i,j}=\begin{cases*}
        (\lambda_{n-k+1-j})+j-i \qquad \text{if $j\leq n-k$}\\
        j-i \qquad \text{if $j>n-k$}.
    \end{cases*}
\end{align*}
Then we have 
\begin{equation*}
     \sum_{S \in \binom{[n]}{k}} \dfrac{s_{\tilde{\lambda}}[z_S]P(z_{S^{c}},z_S)}{\prod_{\substack{i \in S \\ j \in S^c}} (z_j - z_i)}= (-1)^{|\tilde{\lambda}|}\sum_{\sigma\in\mathfrak{S}_n}(-1)^{\sgn(\sigma)}[\prod_{i=1}^{n}z_i^{V_{\sigma(i),i}}]P(z_1,z_2,\dots,z_n)
\end{equation*}
where notation $[z_1^{c_1}z_2^{c_2}\cdots z_n^{c_n}]P(z_1,z_2,\dots,z_n)$ on the right-hand side represents the operation of taking the coefficient.

\end{lem}
\begin{proof}
We denote $<-,->^{(1)}$ (resp $<-,->^{(2)}$) to be the usual Hall inner product on symmetric functions with variables $z_1,z_2,\dots,z_{n-k}$ (resp $z_{n-k+1},z_{n-k+2},\dots,z_{n}$). First we expand $P$ in terms of the Schur functions in variables $z_{n-k+1},z_{n-k+2},\dots,z_{n}$:
\begin{equation*}
    P=\sum_{\mu}Q_{\mu}(z_1,z_2,\dots,z_{n-k})s_{\mu}(z_{n-k+1},z_{n-k+2},\dots,z_{n}).
\end{equation*}
Letting $s_{\tilde{\lambda}}s_{\mu}=\sum_{\nu}c_\nu s_{\nu}$, by Lemma \ref{lem: Schur orthogonal in rectangle} we have
\begin{equation*}
     \sum_{S \in \binom{[n]}{k}} \dfrac{s_{\tilde{\lambda}}(z_S)Q_{\mu}(z_{S^{c}})s_{\mu}(z_{S})}{\prod_{\substack{i \in S \\ j \in S^c}} (z_j - z_i)}=(-1)^{|\tilde{\lambda}|+|\mu|} <Q_{\mu},\sum_{\nu}c_\nu s_{\tilde{\nu}}>^{(1)}=(-1)^{|\tilde{\lambda}|+|\mu|} <Q_{\mu},s_{\tilde{\mu}/(\tilde{\lambda})'}>^{(1)}
\end{equation*}
where the last equality used $\sum_{\nu}c_{\nu}s_{\tilde{\nu}}=s_{\Box}(\omega(s_{\tilde{\lambda}}s_{\mu}))^{\perp}=s_{\tilde{\mu}/(\tilde{\lambda})'}$ for a rectangle $\Box=R(n,k)$.
We may only consider the case when $\mu$ is inside $R(n,n-k)$. Since $Q_{\mu}=<P,s_{\mu}>^{(2)}$, expanding $s_{\mu}$ in terms of $h$ using Jacobi-Trudi formula gives
\begin{equation*}
    Q_{\mu}=\sum_{\tau\in\mathfrak{S}_k}[\prod_{i=1}^{k}z_{n-k+i}^{V'_{\tau(i),i}}]P
\end{equation*}
where $V'=V_{A,[n-k+1,n]}$ is a submatrix for some $A\in \binom{[n]}{k}$. Again using a Jaboci-Trudi formula for $s_{\tilde{\mu}/(\tilde{\lambda})'}$ gives 
\begin{equation*}
    <Q_{\mu},s_{\tilde{\mu}/(\tilde{\lambda})'}>^{(1)}=\sum_{\tau\in\mathfrak{S}_{n-k}}[\prod_{i=1}^{n-k}z_{i}^{V''_{\tau(i),i}}]Q_{\mu}
\end{equation*}
where $V''=V_{A^{c},[1,n-k]}$. Now summing over all possible $\mu$ inside a rectangle $R(n,n-k)$ is the same as summing over all $A\in \binom{[n]}{k}$ and the proof is complete. 
\end{proof}

Now we describe a monomial coefficient of $B=B_{[n-k+1,n]}$. For $R\in \Sub(\Rec^{\leq N},|\lambda|)$ denoted by $R:D\rightarrow \mathbb{Z}_{\geq1}$, we associate a tuple $\eta(R)=(L_1,L_2,\dots,L_n)$ where $L_i$ is a (finite) subset of $\mathbb{Z}_{\geq 0}\times \mathbb{Z}_{\geq 1}$ given by: $(c,d)\in L_i$ if and only if $(c+1,i)\in D$ and $R(c+1,i)=d$. For example, $\eta(R)=(\{(0,3),(1,4)\},\{\},\{\},\{(1,1),(2,1)\})$ for $R$ in Figure \ref{fig: sub filling}. Note that $|\eta(R)_i|=\rsum(R)_i$
\begin{lem}\label{lem: lemmas for reduced assignments}
Given $R\in \Sub(\Rec^{\leq N},|\lambda|)$ and $\gamma(R)=0$, let $\alpha=\rsum(R)$ and $\beta=\csum(R)$. Then
\begin{enumerate}
    \item denoting $\eta(R)=(L_1,L_2,\dots,L_n)$, we have 
$L_i\in \mathbf{C}_{\beta_i}$ if $i\leq n-k$ and $L_i\in \hat{\mathbf{C}}_{\beta_i}$ if $i> n-k$
    \item for $S=\{n-k+1,n-k+2,\dots,n\}$ we have
    \begin{equation*}
        \stat_{(\Rec^{\leq N},g^{z}_{S})}(R)=\frac{ \stat_{(\Rec^{\leq N},g^{z}_{S})}(w_0)}{q^{n|\lambda|}\prod_{i=1}^{N-1}z_{N+n-i}^{\alpha_i}} q^{\dinv(\eta(R))}\prod_{i=1}^{n}z_i^{\beta_i}.
    \end{equation*}
\end{enumerate}
\end{lem}
\begin{proof}
We denote $R:D\rightarrow \mathbb{Z}_{\geq 1}$ for some $D\subseteq \Rec^{\leq N}$ with $|D|=\lambda$.

(1) Let $L_i=\{(a_1,b_1)\prec_{\lex}\dots\prec_{\lex} (a_{\beta_i},b_{\beta_i})\}$. If $a_{\ell}+1<a_{\ell+1}$ then $(a_\ell,b_\ell)\prec_{p}(a_{\ell+1},b_{\ell+1})$. If $a_{\ell}+1=a_{\ell+1}$ and $b_{\ell}<b_{\ell+1}$ then the cell $(a_\ell-1,i)$ is bad. Therefore we have  $b_{\ell}\geq b_{\ell+1}$, thus $(a_\ell,b_\ell)\prec_{p}(a_{\ell+1},b_{\ell+1})$. We conclude $L_i\in \mathcal{P}_{\beta_i}$. If $i>n-k$, $i$-th column of $\Rec^{\leq N}$ does not have a cell in the first row, we have $L_i\in \hat{\mathbf{C}}_{\beta_i}$.

(2) For every $c=(i,j)\in D$, $(i+1,j)\notin \Des_{\Rec^{\leq N}}(\bar{R})$ and we have $g_{S}^{z}(i+1,j)=\frac{z_{N+n-i}}{z_j}$. There are $n$-many cells $u$ such that $(u,c)$ is an attacking pair, so total $n|\lambda|$ many attacking pairs will be lost. Extra contribution among attacking pairs within $R$ will be counted by $\dinv(\eta(R))$. We conclude 
\begin{equation*}
    \maj_{(\Rec^{\leq N},g^{z}_{S})}(\bar{R})=\frac{\maj_{(\Rec^{\leq N},g^{z}_{S})}(w_0)}{\frac{\prod_{i=1}^{N-1}z_{N+n-i}^{\alpha_i}}{\prod_{i=1}^{n}z_i^{\beta_i}}} \qquad \inv_{(\Rec^{\leq N})}(\bar{R})=\frac{\inv_{(\Rec^{\leq N})}(w_0)}{q^{n |\lambda|}}q^{\dinv{\eta(R)}}.
\end{equation*}
\end{proof}
Recall the operators $\mathfrak{h}_m$, $\bar{\mathfrak{h}}_{m}$ and $\hat{\mathfrak{h}}_{m}$ acting on $\mathbb{F}[\mathbf{y}]$, polynomial ring of $\mathbf{y}$ variables (Definition \ref{def: operator definition}). We let $\Phi$ be an evaluation map given by: $y_{i,j}\rightarrow z_{N+n-i-1}^{-1}x_j$ if $0\leq i\leq N-2$ and $y_{i,j}\rightarrow 0$ if $i>N-2$. Then by Lemma \ref{lem: lemmas for reduced assignments} we have
\begin{equation*}
    [\prod_{i=1}^{n}z_i^{\beta_i}]B=\frac{1}{q^{n|\lambda|}}\Phi \left(\prod_{i=1}^{n-k}\mathfrak{h}_{\beta_i}\prod_{i=n-k+1}^{n}\hat{\mathfrak{h}}_{\beta_i} \cdot 1 \right).
\end{equation*}
Note that the specialization $y_{i,j}\rightarrow 0$ if $i>N-2$ is because $\Rec^{\leq N}$ has $N$ rows and $R$ with $\gamma(R)=0$ is not allowed to assign a top row cell. By Lemma \ref{lem: key connection JT formula}, \eqref{eq: rec final2} equals
\begin{equation}\label{eq: eq rec final3}
    \frac{C_1}{q^{n|\lambda|}}\Phi\left(\det(W) \cdot 1 \right).
\end{equation}
where $W=W^{(s)}$ ($k$ is replaced by $n-k$) defined in \eqref{eq: Wr matrix formula}.

\begin{proof}[Proof of Theorem \ref{thm: main theorem} (c)] In the expansion $\det(W) \cdot 1$, a variable $y_{i,j}$ for $i>N-2$ does not appear as $N$ is sufficiently large. We may only consider the specialization $y_{i,j}\rightarrow z_{N+n-i-1}^{-1}x_j$ and again specializing with \eqref{eq: z_i values}, gives $y_{i,j}\rightarrow q^{n} t^{i}x_j$. Note that $\det(W) \cdot 1$ is of (homogenous) degree $|\lambda|$ in $\mathbf{y}$'s, comparing with \eqref{eq: lw formula final} and using Lemma \ref{lem: for the longest permutation} (3) we have 
\begin{equation*}
    \eqref{eq: eq rec final3} \xrightarrow[\text{Specialize by \eqref{eq: z_i values}}]{\text{}} \frac{T_{\delta^{[n]}} (q^{n})^{|\lambda|}}{q^{n|\lambda|}}\left(\det(W) \cdot 1 \bigg|_{y_{i,j} = t^i x_j}\right)=T_{\delta^{[n]}} (-1)^{\adj(\lambda)}\LW_{\lambda}.
\end{equation*}
Together with \eqref{eq: up toconstants}, the proof is complete.
\end{proof}

\section{Macdonald piece polynomial $\I_{\mu\,\lambda,k}$ at $q=t=1$}
\label{sec: q=t=1}
For a partition $\lambda$ such that  $\ell(\lambda)\leq k$, let $ s^{\flat}_{\lambda}(x_1,x_2,\dots,x_k):=s_{\lambda}(x_1-1,x_2-1,\dots,x_k-1)$. Then we have 
\begin{equation}\label{eq: schur flat expansion}
    s_{\lambda}(x_1,x_2,\dots,x_k)=\sum_{\nu\subseteq \lambda} 
    d^{(k)}_{\lambda,\nu}
     s_{\nu}(x_1-1,x_2-1,\dots,x_k-1)
\end{equation}
where $d^{(k)}_{\lambda,\mu}=\det\left(\binom{\lambda_i+k-i}{\mu_j+k-j}\right)_{1\leq i,j\leq k}$ \cite{Sta16}.
Note that $d^{(k)}_{\lambda,\mu}$ is a number of certain non-intersecting lattice paths by  Lindström–Gessel–Viennot lemma, therefore a nonnegative integer.
Lastly, we denote $[h_{\mu}](\nabla s_{\lambda}\vert_{q,t=1})$ by $W_{\lambda,\mu}$.

\begin{proof}[Proof of Theorem~\ref{thm: q=t=1}]
Consider
\[
    \bar{\I}_{\mu,\lambda,k}[X;q,t] := (-1)^{|\lambda|}\sum_{S \in \binom{[n]}{k}} \dfrac{s^{\flat}_\lambda(z_S) \prod_{j \in S^c}{z_j}}{\prod_{\substack{i \in S \\ j \in S^c}} (z_j - z_i)} \widetilde{H}_{\mu^S}[X]
\]
where $z_i=q^{-\coarm(c_i)}t^{-\coleg(c_i)}$. We will prove the following:
\begin{equation}\label{eq: bar I expansion}
    \bar{\I}_{\mu,\lambda,k}[X;1,1]=\sum_{\tau}W_{\tilde{\lambda},\tau}h_{(\tau+1^{|\mu|-k-|\tilde{\lambda}|})
    }.
\end{equation}

Expanding $s^{\flat}_\lambda(z_S)$ in terms of Schur functions as follows
\begin{equation*}
    s^{\flat}_\lambda(z_S)=\sum_{\nu\subseteq \lambda}d'_{\lambda,\nu}s_{\nu}(z_S),
\end{equation*}
we have
\begin{equation*}
    \bar{\I}_{\mu,\lambda,k}[X;q,t] =\sum_{\nu\subseteq \lambda}(-1)^{|\lambda|-|\nu|}d'_{\lambda,\nu}\I_{\mu,\nu,k}[X;q,t].
\end{equation*}
Note that $d'_{\lambda,\lambda}=1$. Therefore, by Theorem \ref{thm: main theorem} (a) and (b) we have
\begin{equation}\label{eq: something}
    e^\perp_N \left(\bar{\I}_{\mu,\lambda,k}[X;q,t]\right) = \begin{cases}
        0 &\text{ if } |\mu|-k-|\tilde{\lambda}|< N \\
        \nabla s_{\tilde{\lambda}} &\text{ if }|\mu|-k-|\tilde{\lambda}|= N.
    \end{cases}
\end{equation}
Together with the following elementary fact
\begin{equation*}
    e^\perp_N h_\nu = \begin{cases}
        0 &\text{ if } \ell(\nu) < N \\
        h_{\nu-(1^N)} &\text{ if } \ell(\nu) = N,
    \end{cases}
\end{equation*}
we conclude 
\begin{equation*}
        [h_{\nu}] \left(\bar{\I}_{\mu,\lambda,k}[X;1,1]\right) = \begin{cases}
        0 &\text{ if } \ell(\nu) >  |\mu|-k-|\tilde{\lambda}| \\
        W_{\tilde{\lambda},\nu'} &\text{ if } \ell(\nu) =  |\mu|-k-|\tilde{\lambda}|,
    \end{cases}
\end{equation*}
where $\nu'=\nu-(1^{ |\mu|-k-|\tilde{\lambda}|})$.

Now it suffices to show $[h_{\nu}] \left(\bar{\I}_{\mu,\lambda,k}[X;1,1]\right)=0$ for $\ell(\nu)<|\mu|-k-|\tilde{\lambda}|$. Recall that for a partition $\mu\vdash n$, we have,
\[
    \widetilde{H}_\mu[X;q,1]=\sum_{\nu\vdash n} (q-1)^{n-\ell(\nu)}A_{\mu,\nu}(q)h_\nu[X],
\]
for some $A_{\mu,\nu}(q)\in\mathbb{Z}[q]$ \cite[Proposition 1.1]{GHQR19}. From the definition of $\bar{\I}_{\mu,\lambda,k}[X;q,t]$, we have,
\[
    [h_\tau]\left(\bar{\I}_{\mu,\lambda,k}[X;q,1]\right)=(-1)^{|\lambda|}\sum_{S \in \binom{[n]}{k}} \dfrac{s^{\flat}_\lambda(z_{S}) \prod_{j \in S^c}{z_j}}{\prod_{\substack{i \in S \\ j \in S^c}} (z_j - z_i)} (q-1)^{|\mu|-k-\ell(\nu)}A_{\mu^{S},\tau}(q),
\]
where we abuse our notation to write $z_i=q^{-\coleg(c_i)}$. Note that
\begin{equation*}
    \dfrac{s^{\flat}_\lambda(z_{S}) \prod_{j \in S^c}{z_j}}{\prod_{\substack{i \in S \\ j \in S^c}} (z_j - z_i)}
\end{equation*}
has a pole of degree $k(n-k)-|\lambda|=|\tilde{\lambda}|$ at $q=1$. Since $|\mu|-k-\ell(\tau)>|\tilde{\lambda}|$, each term 
\[
    \dfrac{s^{\flat}_\lambda(z_{S}) \prod_{j \in S^c}{z_j}}{\prod_{\substack{i \in S \\ j \in S^c}} (z_j - z_i)} (q-1)^{|\mu|-\ell(\nu)}A_{\mu^{S},\tau}(q)
\]
vanishes at $q=1$.

By \eqref{eq: schur flat expansion} we have
 \begin{equation*} \I_{\mu,\lambda,k}[X;1,1] =\sum_{\nu\subseteq \lambda}(-1)^{|\lambda|-|\nu|}d^{(k)}_{\lambda,\nu}(\bar{\I}_{\mu,\nu,k}[X;1,1]).\end{equation*}
Now \eqref{eq: bar I expansion} completes the proof. 
\end{proof}

\section{Future Direction}
\label{sec: future questions}
We list some future questions from this work.
\begin{itemize}
\item We conjectured that \(\I_{\mu,\lambda,k}\) is signed Schur positive and \(\widetilde{H}^{\lambda}_{\mu^S}\) is Schur positive. Therefore, finding positive monomial expansion formulas (possibly of the HHL type) for these symmetric functions would be the next challenge. As we have seen that \(\widetilde{H}^{\lambda}_{\mu^S}\) refines the Macdonald positivity, finding a positive monomial expansion would be a significant step toward addressing Problem \ref{prob: Macdonald-Kostka}. Additionally, we may further attempt to extend Conjecture \ref{conj: sf extend}, by expressing not only the intersection but also the intersection together with the difference of Garsia--Haiman modules.

\item The arguments presented in Section \ref{sec: LW formula} remain valid when we change a poset $\mathbf{P}$ to any unit interval order. It would be interesting to find a generalization of $\I_{\mu,\lambda,k}$ that matches with the combinatorics of the generalized $\LW_{\lambda}$. 

\item The Macdonald polynomials have fascinating connections to the geometry of Hilbert schemes \cite{Hai01} and affine Springer fibers \cite{Mel20}. However, understanding our main results from a geometric perspective remains elusive. We hope that our findings provide new insights and contribute to the theory of Macdonald polynomials.

\end{itemize}

\appendix
\section{Data for the dimension of $\widetilde{H}^{\lambda}_{\mu^{S}}$}\label{sec: Appendix}
\label{Sec: Appendix}
Recall the symmetric function \(\widetilde{H}^{\lambda}_{\mu^{S}}\) defined in \eqref{eq: module intersection sym def}. Note that we can express \(\widetilde{H}^{\lambda}_{\mu^{S}}\) in terms of \(\I_{\mu,\nu,k}\)'s, and each \(\I_{\mu,\nu,k}\) has a fixed \(h\)-expansion formula (up to a multiplication by \(h_{(1^m)}\)) at \(q=t=1\) by Theorem \ref{thm: q=t=1}. Therefore, the 'relative dimension'
\begin{equation*}
    \RD(\lambda) := \frac{\langle \widetilde{H}^{\lambda}_{\mu^{S}}[X;1,1], e_{(1^{|\mu|-k})} \rangle}{(|\mu|-k)!}
\end{equation*}
is a well-defined number that only depends on \(\lambda\). Note that \(\RD((n)) = \frac{1}{n+1}\) as proved in \cite{BG99, KLO23}, which is also compatible with \eqref{eq: n!/k conjecture}. We provide experimental data for \(\RD(\lambda)\) for \(|\lambda| \leq 6\). According to Conjecture~\ref{conj: sf extend}, these numbers are the relative dimension of a certain intersection of Garsia--Haiman modules.

\begin{table}[h]
        \begin{tabular}{|c|c|c|c|c|c|c|c|c|c|c|c|}
        \hline
        $\lambda$ & $(1)$ & $(2)$ & $(1,1)$ & $(3)$ & $(2,1)$ & $(1,1,1)$ & $(4)$ &$(3,1)$ &$(2,2)$ & $(2,1,1)$ & $(1,1,1,1)$  \\\hline
        $\RD(\lambda)$ & $\frac{1}{2}$ & $\frac{1}{3}$ & $\frac{5}{12}$& $\frac{1}{4}$ &$\frac{1}{4}$ &$\frac{3}{8}$ &$\frac{1}{5}$ & $\frac{7}{40}$ &$\frac{17}{72}$ & $\frac{77}{360}$ &$\frac{251}{720}$\\\hline
        $\lambda$ & $(5)$ & $(4,1)$ & $(3,2)$ & $(3,1,1)$ & $(2,2,1)$ & $(2,1,1,1)$ & $(1,1,1,1,1)$ \\\cline{1-8}
        $\RD(\lambda)$ & $\frac{1}{6}$ & $\frac{2}{15}$ & $\frac{19}{120}$& $\frac{7}{48}$ &$\frac{47}{240}$ & $\frac{139}{720}$ &$\frac{95}{288}$ \\\cline{1-8}
        \end{tabular}
\end{table}

\begin{table}[h]
        \begin{tabular}{|c|c|c|c|c|c|c|c|c|c|c|c|}
        \hline
        $\lambda$ & $(6)$ & $(5,1)$ & $(4,2)$ & $(4,1,1)$ & $(3,3)$ & $(3,2,1)$ & $(3,1,1,1)$  \\\hline
        $\RD(\lambda)$ & $\frac{1}{7}$ & $\frac{3}{28}$ & $\frac{7}{60}$ & $\frac{23}{210}$ &$\frac{37}{240}$ &$\frac{1}{8}$ & $\frac{437}{3360}$ \\\hline
        $\lambda$ & $(2,2,2)$ & $(2,2,1,1)$ & $(2,1,1,1,1)$ & $(1,1,1,1,1,1)$  \\\cline{1-5}
        $\RD(\lambda)$ & $\frac{413}{2160}$ & $\frac{749}{4320}$ &$\frac{5419}{30240}$ &  $\frac{19087}{60480}$ \\\cline{1-5}
        \end{tabular}
\end{table}

For the final remarks, we display a series of observations, which we do not have a proof, according to the data above.
\begin{enumerate}
    \item (Pieri rule) For a partition $\lambda$, let $\lambda^+$ be the set of partitions obtained from $\lambda$ by adding a cell. Then we have
    \[
        \left(|\lambda^+|-\frac{1}{2}\right)\RD(\lambda) = \sum_{\mu \in \lambda^+} \RD(\mu).
    \]
    
    \item (one column shapes) We have
    \[
        \RD((1^{n-1})) = \dfrac{N_n}{H_n},
    \]
    where the numerator is the \emph{Nörlund number} given in \cite[\href{https://oeis.org/A002657}{A002657}]{Sloane}. and the denominator is the \emph{Hirzebruch number} given in \cite[\href{https://oeis.org/A091137}{A091137}]{Sloane}.
    \item (staircase shapes) Let $\text{st}_n=(n-1,n-2,\dots,1)$ be the staircase partition. Then we have
    \[
    \RD(\text{st}_n)=\dfrac{1}{2^n}.
    \]
    \item (sum of hooks) Let $\operatorname{Hook}_n$ be the set of hook shape partition of size $n$. Then we have
    \[
    \sum_{\lambda\in\operatorname{Hook}_n} \RD(\lambda ) = 1- \dfrac{1}{2^n}.
    \]
\end{enumerate}

\bibliographystyle{alpha}  
\bibliography{main.bib} 

\end{document}